\documentclass[onefignum,onetabnum]{siamonline220329}

\usepackage{amsmath}
\usepackage{thmtools}
\usepackage{hyperref}
 
\theoremstyle{plain}
\theoremheaderfont{\normalfont\bfseries}
\theorembodyfont{\normalfont\itshape}
\theoremseparator{.}
\theoremsymbol{}
\newtheorem{namedTheorem}[theorem]{Theorem}
\newtheorem{namedLemma}[theorem]{Lemma}
\newtheorem{namedAssumption}[theorem]{Assumption}

\usepackage{lipsum}
\usepackage{amsfonts}
\usepackage{graphicx}
\usepackage{epstopdf}
\usepackage{algorithmic}
\usepackage[normalem]{ulem}  

\usepackage{tikz}
\usetikzlibrary{shapes.geometric, arrows.meta, bending, positioning}

\renewcommand{\vec}[1]{\boldsymbol{#1}}
\newcommand{\mat}[1]{\boldsymbol{\mathrm{#1}}}

\newcommand{\fRHS}{f}
\newcommand{\gRHS}{g}

\newcommand{\vbefore}{\vec{v}_j^{l-1}}
\newcommand{\vafter}{\vec{v}_j^l}

\ifpdf
  \DeclareGraphicsExtensions{.eps,.pdf,.png,.jpg}
\else
  \DeclareGraphicsExtensions{.eps}
\fi

\usepackage{enumitem}
\setlist[enumerate]{leftmargin=.5in}
\setlist[itemize]{leftmargin=.5in}

\newsiamremark{remark}{Remark}
\newsiamremark{hypothesis}{Hypothesis}
\crefname{hypothesis}{Hypothesis}{Hypotheses}
\newsiamthm{claim}{Claim}

\usepackage{makecell}
\usepackage{threeparttable}
\newcommand{\norm}[1]{\lVert#1\rVert}
\newcommand{\brac}[1]{\left(#1\right)}
\newcommand{\R}{\mathbb{R}}
\newcommand{\C}{\mathbb{C}}

\headers{Implicitly-templated ODE-nets}{A. Zhu, T. Bertalan, B. Zhu, Y. Tang, and I. G. Kevrekidis}

\title{Implementation and (Inverse Modified) Error Analysis \\ for implicitly-templated ODE-nets}

\author{Aiqing Zhu\thanks{LSEC, ICMSEC, Academy of Mathematics and Systems Science, Chinese Academy of Sciences, Beijing 100190, China (\email{zaq@lsec.cc.ac.cn})}
\and Tom Bertalan\thanks{Department of Chemical and Biomolecular Engineering, Johns Hopkins University, Baltimore, Maryland 21211, USA (\email{tom@tombertalan.com})}
\and Beibei Zhu\thanks{School of Mathematics and Physics, University of Science and Technology Beijing, Beijing 100083, China (\email{zhubeibei@lsec.cc.ac.cn})}
\and Yifa Tang\thanks{Corresponding author, LSEC, ICMSEC, Academy of Mathematics and Systems Science, Chinese Academy of Sciences, Beijing 100190, China (\email{tyf@lsec.cc.ac.cn})}
\and Ioannis G. Kevrekidis\thanks{Corresponding author, Department of Chemical and Biomolecular Engineering and Department of Applied Mathematics and Statistics, Johns Hopkins University, Baltimore, Maryland 21211, USA (\email{yannisk@jhu.edu})}}

\usepackage{amsopn}

\makeatletter
\@ifpackageloaded{hyperref}%
  {\newcommand{\mylabel}[2]
    {\protected@write\@auxout{}{\string\newlabel{#1}{{#2}{\thepage}%
      {\@currentlabelname}{\@currentHref}{}}}}}%
  {\newcommand{\mylabel}[2]
    {\protected@write\@auxout{}{\string\newlabel{#1}{{#2}{\thepage}}}}}
\makeatother

\begin{document}

\maketitle

\begin{abstract}

We focus on learning unknown dynamics from data using ODE-nets templated on implicit numerical initial value problem solvers.
First, we perform Inverse Modified error analysis of the ODE-nets using unrolled implicit schemes for ease of interpretation. It is shown that training an ODE-net using an unrolled implicit scheme returns a close approximation of an Inverse Modified Differential Equation (IMDE). 
In addition, we establish a theoretical basis for hyper-parameter selection when training such ODE-nets, whereas current strategies usually treat numerical integration of ODE-nets as a black box. 
We thus formulate an adaptive algorithm which monitors the level of error and adapts the number of (unrolled) implicit solution iterations during the training process, so that the error of the unrolled approximation is less than the current learning
loss.
This helps accelerate training, while maintaining accuracy. Several numerical experiments are performed to demonstrate the advantages of the proposed algorithm 
compared to nonadaptive unrollings, 
and validate the theoretical analysis.
We also note that this approach naturally allows for incorporating partially known
physical terms in the equations, giving rise to what is termed ``gray box" identification.
\end{abstract}

\begin{keywords}
learning dynamics, deep learning, ODE-nets, implicit schemes, neural ODEs
\end{keywords}

\begin{MSCcodes}
37M10, 65L06, 65L09, 65P99
\end{MSCcodes}

\section{Introduction}
Discovering unknown dynamical systems from observed dynamical data is an established systems task where machine learning has been shown to be remarkably effective.
Neural networks $\fRHS{}_{\theta}$, coined  ``ODE-nets", are used to parameterize the unknown governing differential equations; their parameters $\theta$ are obtained by minimizing the difference between the observed state time series and the outputs evaluated by numerically solving the ODE governed by the right-hand-side $\fRHS{}_{\theta}$. 
Original publications along this line date back to the 1990s~\cite{anderson1996comparison,gonzalez1998identification,rico1994continuous,rico1993continuous}. Recently, Neural ODEs~\cite{chen2018neural} substantially revisited these ideas using modern computational tools, and is being applied to more challenging tasks beyond modeling dynamical systems. Here, the adjoint reverse-time equations---introduced as a continuous-time analogue of backpropagation---are employed for computation of gradients. 
In addition, various related architectures have been proposed \cite{raissi2018multistep, hu2022revealing, wu2019numerical}, and research interest in this direction has been growing to include coupling machine learning with prior knowledge of (some) physics of the underlying systems \cite{bertalan2019learning,botev2021priors, chen2021data,huh2020time,jin2020sympnets,wu2020structure,yu2021onsagernet} (see also \cite{rico1994continuous} for a discussion on gray-box modeling for incorporating known physics into such learned models).

However, even assuming the best-case convergence of the optimizer and accuracy of the data, the numerical integration of the network used to fit the data can itself introduce a bias into the equations extracted. In this paper, we propose to analyse the influence of the numerical integration scheme template in such learning models.

In the last few decades, modified differential equations (MDEs) and backward error analysis \cite{eirola1993aspects,feng1991formal,feng1993formal,hairer1997life,reich1999backward, sanz1992symplectic,yoshida1993recent} have become well-established tools for analyzing the numerical solution of evolution equations (where we produce approximate trajectories from a true ODE). The main idea of MDEs is to interpret the numerical solution as the exact solution of a perturbed differential equation expressed by a formal series. We can then analyze the MDE, which is easier than the analysis of the discrete numerical solution.

Recently, inspired by MDEs and BE, Inverse Modified Differential Equations (IMDEs) \cite{zhu2022on} have been proposed; they allow the efficient analysis of numerical schemes applied to the discovery of dynamics ({\em where we produce an approximate ODE from true trajectories}). By analogy with the MDE (see \cref{fig:diagram}), the IMDE is a perturbed differential equation whose numerical solution matches the exact observed solution (the data). It was shown in \cite{zhu2022on} that training an ODE-net returns a close approximation of the IMDE, and that some known analysis results of solving ODEs, such as order of convergence, have natural extensions to the field of discovery of dynamics.
\begin{figure}[htb]
    \centering
    \includegraphics[width=0.5\linewidth]{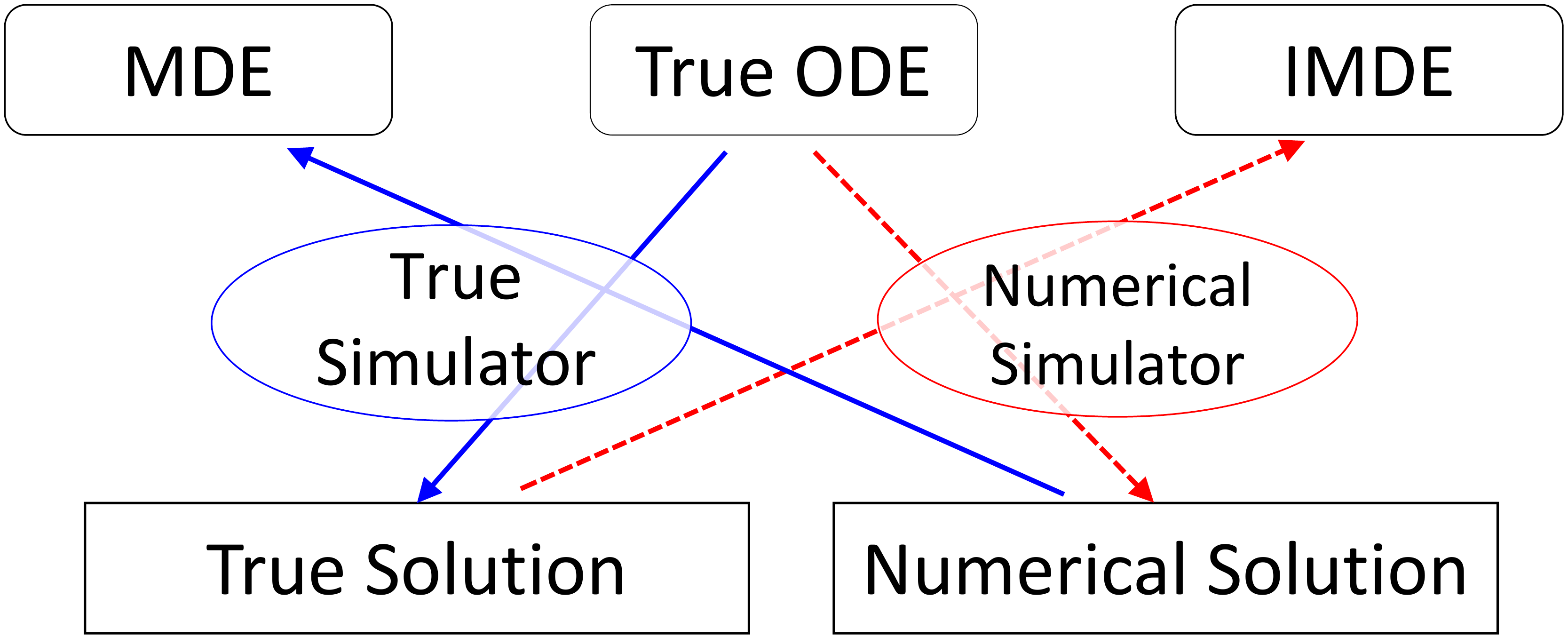}
    \caption{Schematic depiction of the relation between the true model, the Modified and
        the Inverse Modified Differential  Equations (MDE and IMDE respectively). 
        Forward Error Analysis studies the difference between the true and the numerical solution of the model, while Backward, and Inverse Backward Error 
        Analysis examine the difference between the true model and the MDE/IMDE respectively. 
        }
    \label{fig:diagram}
\end{figure}
Other analysis results exist for the discovery
of dynamics by combining numerical integrators and deep learning. In \cite{keller2021discovery}, a refined framework is established to derive the convergence and stability of Linear Multistep Neural Networks (LMNets) \cite{raissi2018multistep} via the characteristic polynomial of classic numerical linear multistep methods (LMM). In addition, an augmented loss function was introduced, based on auxiliary conditions that serve a purpose analogous to the explicit starting step used when performing forward integration with a LMM.
It has been shown that the grid error of the global minimizer is bounded by the sum of the discretization error and the approximation error \cite{du2022discovery}.
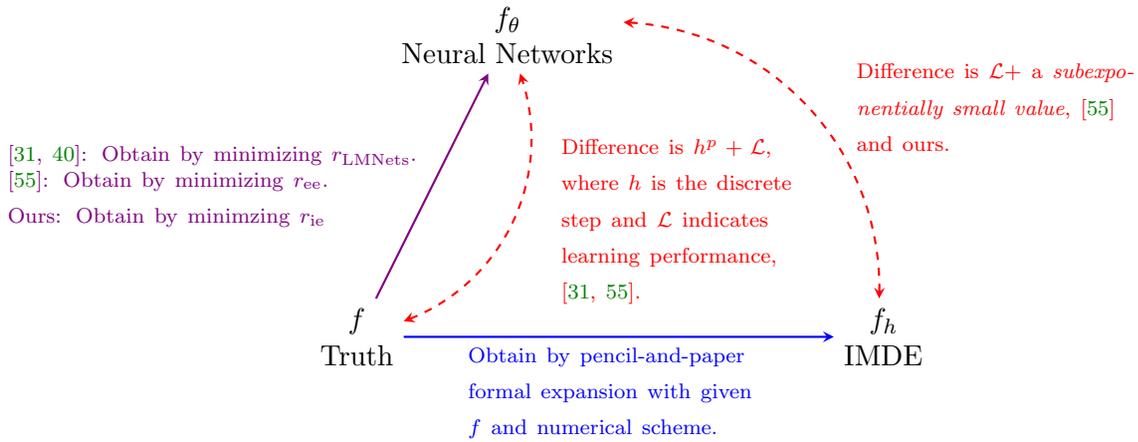
\begin{figure}[htb]
    \centering
    \newcommand{\edgesize}[1]{\scriptsize #1}
    \begin{tikzpicture}[node distance=4cm]
        \tikzstyle{term} = [rectangle, minimum height=1cm, text centered, draw=none, fill=none, align=center]
        \tikzstyle{arrow} = [thick, ->, >=stealth]
        \tikzstyle{errArrow} = [thick, <->, >=stealth, dashed, red]
        \node (ftrue) [term] {$f$ \\ Truth};
        \node (fIMDE) [term,right of=ftrue, xshift=3cm] {$f_h$ \\ IMDE};
        \node (ftheta) [term,above of=ftrue, xshift=2cm] {$f_\theta$ \\ Neural Networks};
        \draw [arrow,blue] (ftrue) -- node[anchor=north, align=left, text width=4cm] {
            \edgesize{Obtain by pencil-and-paper formal expansion with given $f$ and numerical scheme.}
        } (fIMDE);
        \draw [arrow,violet] (ftrue) -- node[anchor=east, align=left, text width=5.5cm] {
            \edgesize{\cite{keller2021discovery,raissi2018multistep}: Obtain by minimizing $r_\mathrm{LMNets}$. \\
            \cite{zhu2022on}: Obtain by minimizing $r_\mathrm{ee}$. \\
            Ours: Obtain by minimzing $r_\mathrm{ie}$
            }
        } (ftheta);
        \draw [errArrow] (ftrue) to [bend right=45] node[anchor=west, align=left, xshift=0.5cm, text width=3.5cm] {
            \edgesize{Difference is $h^p + \mathcal{L}$, where $h$ is the discrete step and $\mathcal{L}$ indicates learning performance, \cite{keller2021discovery,zhu2022on}.}
        } (ftheta);
        \draw [errArrow] (fIMDE) to [bend right=45] node[anchor=west, align=left, text width=4cm, xshift=0.5cm] {
            \edgesize{Difference is $\mathcal{L} +$ a \textit{subexponentially small value}, \cite{zhu2022on} and ours.}
        } (ftheta);
    \end{tikzpicture}
    \caption{
        \textbf{Relationships between different ODEs considered in the backwards analysis literature.}
        A schematic diagram showing existing theoretical analyses and our main results with explicit Euler (ee)  and implicit Euler (ie) schemes as examples.
        The residual of explicit Euler is
        $r_\mathrm{ee} = ||\phi_{\Delta t}(x) - (x + \Delta t \cdot f_\theta(x)||_2^2$
        where $\{x, \phi_{\Delta t}(x)\}$ are the data.
        The residual of the LMNets approach to implicit Euler is
        $r_\mathrm{LMNets} = ||\phi_{\Delta t}(x) - (x + \Delta t \cdot f_\theta(\phi_{\Delta t}(x) )||_2^2$, given the data.
        The residual of our implicit Euler is
        $r_\mathrm{ie}=||\phi_{\Delta t}(x) - \mathrm{argsoln}_z \{z = x + \Delta t \cdot f_\theta(z)\}||_2^2$ where $z$ is the network prediction obtained by a root-finding algorithm.
    }
    \label{fig:results}
\end{figure}

These analyses concentrate on LMM in LMNets, where all LMM discretization (typically implicit) can be exactly employed, and directly quantify the error between the true governing function and its neural network approximation. The existence of an associated IMDE implies uniqueness of the solution to the learning task (in a concrete sense), and also allows us to analyze the numerical error in ODE-nets. However, the results in \cite{zhu2022on} only hold when the numerical integration is {\em exactly} evaluated, whereas the implementation of implicit integration in ODE-nets requires a root-finding algorithm, i.e. by unrolling the iterations, so as to obtain an accurate approximate solution. The mutual differences between these existing theoretical analyses and our main results are schematically visualized in \cref{fig:results}.

In this paper, we extend the analysis proposed in~\cite{zhu2022on} and perform IMDE
analysis for ODE-nets in which we unroll (and truncate) the iterations for solving the implicit scheme within the network architecture. 
To begin with, we search for a perturbed differential equation, i.e., the IMDE, such that its unrolled implicit integration matches observations of the exact solution of the true system. It is noted that this IMDE now depends on the number of unrolled stages (iteration number) of the unrolled implicit scheme. In addition, we prove that, under reasonable assumptions, training an ODE-net using an unrolled implicit scheme returns an approximation of the corresponding IMDE. As a direct consequence, increasing the iteration number results in a more accurate recovery of the IMDE. Finally, the rate of convergence of ODE-nets using unrolled implicit schemes is also presented. Several experiments are performed to validate the analysis, and the numerical results are in agreement the theoretical findings.

The numerical integration of ODE-nets is typically treated as a black box in current strategies.
Here, an unrolling approach to implicit integration requires recurrent calculations; augmenting computational cost and,
in particular, memory demands.
Based on the analysis results, we establish a theoretical basis for hyper-parameter selection when training ODE-nets. We formulate an adaptive algorithm that monitors the level of error and adapts the iteration number in the training process to accelerate training while maintaining accuracy. In the initial stage of training, a rough approximation target, i.e., a smaller iteration number, is accurate enough for optimization. As learning loss decreases, we increase the iteration number so as to achieve a more accurate target. Numerical experiments show that the proposed algorithm leads to a $2$-$3\times$ speedup in training without any degradation in accuracy.

\subsection{Related works}
There have been extensive attempts to determine unknown dynamics using various approaches including symbolic regression~\cite{schmidt2009distilling},
Gaussian processes~\cite{raissi2018hidden},
sparse regression~\cite{brunton2016discovering}, statistical learning~\cite{lu2019nonparametric}, etc. 
Among various models, the ODE-nets~\cite{rico-martinez92_discr_vs,chapter_bulsari,chen2018neural,rico1993continuous}
have been established as powerful tools to model complicated physical phenomena from time series data, and have achieved numerous successes~\cite{anderson1996comparison, botev2021priors, gonzalez1998identification, huh2020time, raissi2018multistep,rico1994continuous,rico1993continuous}.
Recently, researchers have focused on leveraging a continuous-time representation to incorporate physical inductive biases such as symplectic structure~\cite{bertalan2019learning,greydanus2019hamiltonian,wu2020structure}, the Onsager principle~\cite{yu2021onsagernet}, the GENERIC formalism~\cite{zhang2021gfinns} and time-reversal symmetry~\cite{huh2020time}, to name a few, into the learning model.

The implementation of ODE-nets and their variants is inevitably linked with numerical integration.
Several libraries such as \textit{torchdiffeq}, \textit{diffrax} and \textit{torchdyn} have been developed to provide standardized differentiable implementations of ODE solvers. 
Many learning models use the Euler discretization method (e.g.~\cite{bertalan2019learning,greydanus2019hamiltonian}) or higher-order explicit Runge-Kutta methods (e.g.~\cite{yu2021onsagernet}),
while some models encoding symplecticity use a symplectic integrator to preserve the special Hamiltonian form (e.g.,~\cite{chen2020symplectic,toth2020hamiltonian}). 
The work in~\cite{pal2021opening} proposed a novel 
stiffness regularization
for ODE-nets based on the internal cost of numerical integration. 
The interplay between learning Neural ODEs and numerical integration is explored in~\cite{poli2020hypersolver}, where so-called hypersolvers are introduced for fast inference.
A comprehensive study of 
{gradient-based learning} with
implicit integration was explored in previous work~\cite{anderson1996comparison},
{considering unrolling as well as
Pineda's and Almeida's Recurrent Back-Propagation~\cite{pineda87_gener_back_propag_to_recur_neural_networ,Almeida1987}}. In this paper, we focus on the implementation of
{\em unrolled}
implicit numerical integration within ODE-nets, its numerical analysis, and the adaptation of the iteration number for the solution of the implicit problem to reduce computational cost.

Recent works~\cite{bai2019deep,behrmann2019invertible,huang2021implicit,laurent2021implicit} proposed various versions of implicit models and demonstrated their empirical success;
they directly exploit root-finding algorithms (e.g., fixed-point iteration, Newton-Raphson iteration and its variants, Broyden’s method and Anderson acceleration)
to solve for the output in the forward pass.
In \cite{bai2022neural} an auxiliary network was introduced, to provide both the initial value and perform iterative updates to improve inference efficiency.
In \cite{geng2021on} a novel gradient estimate was proposed, to circumvent computing the exact gradient by implicit differentiation. Although the precise formulations and motivations of these implicit models are quite different, applications of our adaptive algorithm to these implicit models is a promising avenue for future work.

\section{Problem setup}\label{sec: problem setup}
Consider the dynamical system
\begin{equation}\label{eq:ODE2}
\frac{d}{dt}\vec{y}(t) = \fRHS{}(\vec{y}(t)),\quad \vec{y}(0) =\vec{x}
\end{equation}
where $ \vec{y}(t) \in \mathbb{R}^D$ is the state vector, evolving over time according to the governing function $\fRHS{}$. Let $\phi_{t}(\vec{x})$ be the exact solution and $\Phi_{h}(\vec{x})$ be the numerical solution (by some initial value problem solving algorithm) with discrete step $h$. In order to emphasize a specific differential equation, we will add the subscript $\fRHS{}$ and denote $\phi_t$ as $\phi_{t,\fRHS{}}$ and $\Phi_h$ as $\Phi_{h,\fRHS{}}$.

If $\fRHS{}$ and the initial state $\vec{x}$ are known, the future states can be predicted by solving the equation \cref{eq:ODE2}. On the other hand, if the exact governing equation is unknown, but some trajectories are given, ODE-nets model the dynamical system by neural networks and then predict future states via the learned model.

Mathematically, an ODE-net identified right-hand-side leads to the ODE
\begin{equation}\label{eq:ODE-net}
\frac{d}{dt}\tilde{\vec{y}}(t) = \fRHS{}_{\theta}(\tilde{\vec{y}}(t)), \quad \tilde{\vec{y}}(0) = \vec{x},
\end{equation}
where $\fRHS{}_{\theta}$ is the neural network approximating the unknown vector field $\fRHS{}$.
With initial condition $\vec{x}$, an ODE-net predicts the output by solving \cref{eq:ODE-net} numerically. 
Given $N$ observed trajectories $\vec{x}_n, \phi_{\Delta t}(\vec{x}_n), \cdots, \phi_{M\Delta t}(\vec{x}_n)$, $n=1,\cdots, N$ with time step $\Delta t$,
the network parameters are determined by minimizing the loss function
\begin{equation}\label{eq:loss}
\mathcal{L}_{exact} = \sum_{n=1}^N \sum_{m=1}^M\|\text{ODESolve}(\vec{x}_n, \fRHS{}_{\theta}, m \Delta t) - \phi_{ m \Delta t}(\vec{x}_n)\|_2^2/(m \Delta t)^2.
\end{equation}
Note that variable data step $\Delta t$ are also possible.
$M=1$ is the classical ``teacher forcing"; excessive M
can be both computationally costly and offer limited benefits especially early in training, when the long-time predictions are especially poor.
So, if the training data is in the form of long trajectories,
we often divide them into smaller sub-episodes, leading 
to an $M$-step teacher forcing scheme \cite{williams1989learning}.
We used $M=1$ for all of the numerical experiments in this paper except the last, for which we used $M=10$.

In this paper, the choice for the ODE solver consists of $s$ compositions of a numerical scheme, i.e.,
\begin{equation*}
\text{ODESolve}(\vec{x}, \fRHS{}_{\theta}, m\Delta t)=\underbrace {\Phi_{h,\fRHS{}_{\theta}} \circ \cdots \circ \Phi_{h,\fRHS{}_{\theta}}}_{\text{ $ms$ compositions}}(\vec{x})= \brac{\Phi_{h,\fRHS{}_{\theta}}}^{ms}(\vec{x}),
\end{equation*}
where $h=\Delta t/s$ is the discrete step.
A common choice of numerical scheme $\Phi_h$ is the Runge-Kutta method, 
which is formulated as
\begin{subequations}\label{eq:rk}
\begin{align}
&\vec{v}_i = \vec{x} + h\sum_{j=1}^I a_{ij}  \fRHS{}_{\theta}(\vec{v}_j) \quad i = 1, \cdots, I\label{eq:rk1}\\
&\Phi_{h,\fRHS{}_{\theta}}(\vec{x}) = \vec{x} + h\sum_{i=1}^I b_{i}  \fRHS{}_{\theta}(\vec{v}_i).
\end{align}
\end{subequations}
A Runge-Kutta method \cref{eq:rk} is explicit only if $a_{ij}=0$ for $i\leq j$.
Otherwise it is implicit, and the output has to be computed iteratively.
For example, we could use fixed-point iteration (successive substitution) with fixed iteration number $L$, 
{in which case} the approximation of \cref{eq:rk},
denoted by $\Phi_{h,\gRHS{}}^L(\vec{x})$,
is given by
\begin{equation}\label{eq:irk}
\begin{aligned}
&\vec{v}_i^{0} = \vec{x}\quad i = 1, \cdots, I,\\
&\vec{v}_i^{l} = \vec{x} + h\sum_{j=1}^I a_{ij}  \fRHS{}_{\theta}(\vec{v}_{j}^{l-1}) \quad i = 1, \cdots, I, \ l = 1, \cdots, L.\\
&\Phi_{h,\fRHS{}_{\theta}}^L(\vec{x}) = \vec{x} + h\sum_{i=1}^I b_{i} \fRHS{}_{\theta}(\vec{v}_i^L).
\end{aligned}
\end{equation}
Newton-Raphson iteration is available as an alternative approach for solving the implicit equation \cref{eq:rk1}, where the approximation using $L$ iterations of \cref{eq:rk}, denoted as $\Phi_{h,\gRHS{}}^L(\vec{x})$, is given by
\begin{equation}\label{eq:irk_nr}
\begin{aligned}
&\vec{v}_i^{0} = \vec{x}\quad i = 1, \cdots, I,\\
&\vec{v}_i^{l} = \vec{x} + h\sum_{j=1}^I a_{ij} \big( \fRHS{}_{\theta}(\vec{v}_{j}^{l-1}) + \fRHS{}_{\theta}'(\vec{v}_{j}^{l-1})(\vec{v}_{j}^{l} - \vec{v}_{j}^{l-1})\big) \quad i = 1, \cdots, I, \ l = 1, \cdots, L.\\
&\Phi_{h,\fRHS{}_{\theta}}^L(\vec{x}) = \vec{x} + h\sum_{i=1}^I b_{i} \fRHS{}_{\theta}(\vec{v}_i^L).
\end{aligned}
\end{equation}
The second of the equations in \cref{eq:irk_nr} is equivalent to the $l^\mathrm{th}$ Newton step, where we know $\vbefore$ for $j=1,\ldots,I$, and we solve (for each $l$) $D\times I$ linear equations $F'(\vec{V}^{l-1}) \cdot (\vec{V}^{l-1} - \vec{V}^{l}) = F(\vec{V}^{l-1})$ to obtain $\vafter$ for $j=1,\ldots,I$. Specifically, 
\begin{equation*}
    \vec{V}^{l-1} = \begin{pmatrix}\vec{v}_1^{l-1}\\ \vec{v}_2^{l-1}\\ \vdots\\ \vec{v}_I^{l-1} \end{pmatrix},\quad 
    \vec{V}^{l} = \begin{pmatrix}\vec{v}_1^{l}\\\vec{v}_2^{l}\\ \vdots\\ \vec{v}_I^{l} \end{pmatrix},\quad
    F(\vec{V}^{l-1}) = \begin{pmatrix}\vec{v}_1^{l-1}- \vec{x} - h\sum_{j=1}^I a_{1j}  \fRHS{}_{\theta}(\vec{v}_j^{l-1})\\ \vec{v}_2^{l-1}- \vec{x} - h\sum_{j=1}^I a_{2j}  \fRHS{}_{\theta}(\vec{v}_j^{l-1})\\ \vdots\\ \vec{v}_I^{l-1} -\vec{x} - h\sum_{j=1}^I a_{Ij}  \fRHS{}_{\theta}(\vec{v}_j^{l-1}) \end{pmatrix},
\end{equation*}
and
\begin{equation*}
F'(\vec{V}^{l-1}) = \begin{pmatrix}\mat{I}_{D\times D}-h a_{11}  \fRHS{}'_{\theta}(\vec{v}_1^{l-1}) & -h a_{12}  \fRHS{}'_{\theta}(\vec{v}_2^{l-1}) & \cdots & -h a_{1I}  \fRHS{}'_{\theta}(\vec{v}_I^{l-1})\\
-h a_{21}  \fRHS{}'_{\theta}(\vec{v}_1^{l-1}) & \mat{I}_{D\times D}-h a_{22}  \fRHS{}'_{\theta}(\vec{v}_2^{l-1}) & \cdots & -h a_{2I}  \fRHS{}'_{\theta}(\vec{v}_I^{l-1})\\ 
\vdots & \vdots & \ddots & \vdots \\
-h a_{I1}  \fRHS{}'_{\theta}(\vec{v}_1^{l-1}) & -h a_{I2}  \fRHS{}'_{\theta}(\vec{v}_2^{l-1}) &\cdots & \ \mat{I}_{D\times D} -h a_{II}  \fRHS{}'_{\theta}(\vec{v}_I^{l-1}) \end{pmatrix}.
\end{equation*}

In either case, (\cref{eq:irk} or \cref{eq:irk_nr}), all of the operations
(including any Jacobian evaluations or inversions) 
have forward and backward implementations in 
established automatic differentiation packages,
and the practical loss function we optimize
is given as
\begin{equation}\label{eq:unrolledloss}
\mathcal{L}_{unrolled}:=\sum_{n=1}^N\sum_{m=1}^{M} \norm{\brac{\Phi_{h,\fRHS{}_{\theta}}^L}^{ms}(\vec{x}_n)- \phi_{m \Delta t}(\vec{x}_n)}_2^2/(m\Delta t)^2.
\end{equation}

\section{Inverse Modified Error Analysis}
The discovery of dynamics using ODE-nets is essentially an inverse process. 
As (direct) Modified Differential Equations (MDEs) were well-established for the numerical analysis of differential equations, 
the idea of a formal extension in Inverse Modified Differential Equations (IMDEs) should prove particularly useful to the study error analysis for ODE-nets~\cite{zhu2022on}. 
In this section, we will extend the results in~\cite{zhu2022on} to unrolled implicit schemes. 

\subsection{Inverse Modified Differential Equations of unrolled implicit schemes}
\label{sec:IMDEunrolledImplicit}
An IMDE
is a perturbed differential equation of the form
\begin{equation*}
\frac{d}{dt}\vec{\tilde{y}}(t)=\fRHS{}_h(\vec{\tilde{y}}(t))  = \fRHS{}_0(\vec{\tilde{y}})+h\fRHS{}_1(\vec{\tilde{y}})+h^2\fRHS{}_2(\vec{\tilde{y}})+\cdots,
\end{equation*}
such that formally
\begin{equation}\label{eq:flow}
\Phi_{h,\fRHS{}_h}(\vec{x}) = \phi_{h,\fRHS{}}(\vec{x}),
\end{equation}
where the identity is understood in the sense of the formal power series in $h$. 
To obtain $\fRHS{}_h$ of an unrolled implicit scheme \cref{eq:irk}, we can expand both sides of \cref{eq:flow} into the corresponding Taylor series around $h = 0$.
First,
\begin{equation}\label{eq:exasolu}
\begin{aligned}
\phi_{h,\fRHS{}}(\vec{x})&=\vec{x}+h\fRHS{}(\vec{x})+\frac{h^2}{2}\fRHS{}'\fRHS{}(\vec{x})+\frac{h^3}{6}(\fRHS{}''(\fRHS{},\fRHS{})(\vec{x})+\fRHS{}'\fRHS{}'\fRHS{}(\vec{x}))\\
&+ \frac{h^4}{24}(\fRHS{}'''(\fRHS{},\fRHS{},\fRHS{})(\vec{x})+3\fRHS{}''(\fRHS{}'\fRHS{},\fRHS{})(\vec{x})+\fRHS{}'\fRHS{}''(\fRHS{},\fRHS{})(\vec{x})+\fRHS{}'\fRHS{}'\fRHS{}'\fRHS{}(\vec{x})) +\cdots .
\end{aligned}
\end{equation}
Here, $\fRHS{}'(\vec{x})$ is a linear map (the Jacobian); the second order derivative $\fRHS{}''(\vec{x})$ is a symmetric bilinear map; and so on for higher order derivatives described as tensors.
We remark that a general expansion \cref{eq:exasolu} can be obtained by Lie derivatives. Next, we expand the  unrolled implicit scheme \cref{eq:irk} as
\begin{equation}\label{eq:numsolu}
\Phi_{h,\fRHS{}_h}^L(\vec{x}) = \vec{x} + hd_{1,\fRHS{}_h}(\vec{x}) + h^2d_{2,\fRHS{}_h}(\vec{x}) + h^3d_{3,\fRHS{}_h}(\vec{x}) +  \cdots,
\end{equation}
where the functions $d_{j,\fRHS{}_h}$ are given --and typically composed of $\fRHS{}_h$ and its derivatives--, and can be calculated by applying B-series~\cite{chartier2010algebraic} on equation \cref{eq:irk}. For consistent integrators, we have
\begin{equation*}
d_{1,\fRHS{}_h}(\vec{x}) = \fRHS{}_h(\vec{x}) = \fRHS{}_0(\vec{x})+h\fRHS{}_1(\vec{x})+h^2\fRHS{}_2(\vec{x})+\cdots.
\end{equation*}
Furthermore, in $h^id_{i,\fRHS{}_h}(\vec{x}) $, the powers of $h$ of the terms containing $\fRHS{}_k$ is at least $k+i$. Thus, the coefficient of $h^{k+1}$ in \cref{eq:numsolu} is
\begin{equation*}
\fRHS{}_k+ \cdots,
\end{equation*}
where the ``$\cdots$'' indicates residual terms composed of $\fRHS{}_j$ with $j< k$ and their derivatives. 
A comparison of equal powers of $h$ \cref{eq:exasolu} and \cref{eq:numsolu} then yields recursively the functions $\fRHS{}_k$ in terms of $\fRHS{}$ and its derivatives. 
Some examples are included in \cref{app:Calculation of IMDE} to illustrate this process. Here, we denote the truncation as
\begin{equation*}\label{eq:trunc}
\fRHS{}_h^K(\vec{y}) = \sum_{k=0}^K h^k \fRHS{}_k(\vec{y}).
\end{equation*}
The IMDE is obtained by paper-and-pencil formal expansion given $\fRHS{}$ and a numerical scheme of choice, and thus is inaccessible in practice due to the unknown true governing function. 
Nevertheless, we will be able to conclude the uniqueness of the solution of the learning task and analyse the numerical integration in ODE-nets.

\subsection{Main results}
We now show that, under reasonable assumptions, training an ODE-net using an unrolled implicit scheme returns a close approximation of the IMDE for the underlying numerical method.

We first set some notation: For a compact subset $\mathcal{K} \subset \C^D$ and the complex ball $\mathcal{B}(\vec{x},r) \subset \C^D$ of radius $r>0$ centered at $\vec{x}\in\C^D$, we define the $r$-dilation of $\mathcal{K}$ as $\mathcal{B}(\mathcal{K}, r) = \bigcup_{x \in \mathcal{K}} \mathcal{B}(\vec{x},r)$.
We will work with $l_{\infty}$- norm on $\C^D$, denote $\norm{\cdot} = \norm{\cdot}_{\infty}$, and for an analytic vector field $\fRHS{}$, define
\begin{equation*}
\norm{\fRHS{}}_{\mathcal{K}} = \sup_{x\in\mathcal{K}}\norm{\fRHS{}(\vec{x})}.
\end{equation*}
Now we present the main result, which implies that the unrolled implicit ODE-net approximates the IMDE.

\begin{namedTheorem}[\textbf{The unrolled approximation approaches the IMDE}]
\label{the:implicit imde}
Consider the dynamical system \cref{eq:ODE2}, a consistent implicit Runge-Kutta scheme $\Phi_h$ \cref{eq:rk},
and its unrolled approximation $\Phi_h^{L}$ (\cref{eq:irk} or \cref{eq:irk_nr}).
Let $\fRHS{}_{\theta}$ be the network learned by optimizing \cref{eq:unrolledloss}.
For $\vec{x} \in \R^D$, $r_1, r_2 >0$, we denote
\begin{equation}\label{eq:L}
\mathcal{L} = \norm{\brac{\Phi_{h,\fRHS{}_{\theta}}^{L}}^s -\phi_{sh,\fRHS{}}}_{\mathcal{B}(\vec{x}, r_1)}/\Delta t,
\end{equation}
and suppose the true vector field $\fRHS{}$ and the learned vector field $\fRHS{}_{\theta}$ are analytic and 
satisfy $\norm{\fRHS{}}_{\mathcal{B}(\vec{x},r_1+r_2)}\leq m, \norm{\fRHS{}_{\theta}}_{\mathcal{B}(\vec{x},r_1+r_2)} \leq m$.
Then, there exists a uniquely defined vector field $\fRHS{}_h^K$, i.e., the truncated IMDE of $\Phi_{h}^L$, such that, if $0<\Delta t<\Delta t_0$,
\begin{equation}\label{eq:error}
\begin{aligned}
&\norm{\fRHS{}_{\theta}(\vec{x}) - \fRHS{}_h^K(\vec{x})} \leq c_1 m e^{-\gamma_1/\Delta t^{1/q}} + \frac{e}{e-1} \mathcal{L},\\
\end{aligned}
\end{equation}
where the integer $K=K(h)$ and the constants $\Delta t_0$, $q$, $\gamma_1$, $c_1$ depend only on $m/r_1$, $m/r_2$, $s$, $\Phi_h$ and the implicit solver\footnote{
The constants here depend on the choice of solver
(specifically, on the constants $b_1, b_2, b_3$ in \cref{asm:int}).
However, since the first term in \cref{eq:error} is very small,
the constants contained have little effect on the results.
}

\end{namedTheorem}
\begin{proof}
The proof can be found in \cref{app:the:implicit imde}.
\end{proof}
Here, the first term on the right hand side of \cref{eq:error} is sub-exponentially small. The $\mathcal{L}$ defined in \cref{eq:L} can be regarded as a form of generalization of the learning loss \cref{eq:loss} when $M=1$ (loss \cref{eq:loss}; with different $M$ we have equivalent convergence due to the following \cref{lem:loss explain}). In this paper we mainly focus on numerical schemes, and thus we will not further quantify $\mathcal{L}$. Provided we make the additional assumption that there are sufficient many data points, the network is sufficiently large and the training finds a neural network with perfect performance, then the learning loss converges to zero and the difference between the learned ODE and the truncated IMDE converges to near-zero (as per \cref{eq:error}). We therefore claim that $\fRHS{}_{\theta}$ is a close approximation of $\fRHS{}_h^K$.

Next, we show that the teacher-forcing loss (i.e., setting $M=1$ in $\mathcal{L}_{unrolled}$ of \cref{eq:unrolledloss}) is bounded by the $M$-step shooting loss on the same data, and thus these two have the equivalent convergence.
\begin{namedLemma}[\textbf{The $M$-step shooting loss and the teacher-forcing loss have equivalent convergence}]\label{lem:loss explain}
Let $\mathcal{T} = \{\phi_{m\Delta t}(\vec{x}_n)\}_{1\leq n \leq N, 0\leq m \leq M-1}$ be the total observed data, then, there exist constants $C_1$, $C_2$, such that
\begin{equation}
\label{eq:ShootingVsTeacherThm}
\begin{aligned}
&C_1\cdot \mathcal{L}_{unrolled}\leq \sum_{x\in\mathcal{T}}\norm{\brac{\Phi_{h,\fRHS{}_{\theta}}^{L}}^s(\vec{x}) -\phi_{sh,\fRHS{}}(\vec{x})}^2_2/\Delta t^2 \leq C_2\cdot \mathcal{L}_{unrolled}.
\end{aligned}
\end{equation}
\end{namedLemma}
\begin{proof}
The proof can be found in \cref{app:lem:loss explain}.
\end{proof}
Since we consider variable $M$ in \cref{sec: problem setup}, 
we perform the analysis that follows for $M=1$,
and use \cref{lem:loss explain} to extend to different choices of $M$.

Next, we have the following \cref{the:hat}, which indicates that increasing the iteration number $L$ is equivalent to adjusting the approximation target to gradually approach the true target with the help of \cref{the:implicit imde}.
\begin{namedTheorem}[\textbf{Increasing the iteration number $L$ is equivalent to adjusting the approximation target to gradually approach the true target}]\label{the:hat}
Consider a consistent implicit Runge-Kutta scheme $\Phi_h$ \cref{eq:rk} and denote the IMDE\footnote{If we suspect that this sum does not converge with $L\rightarrow \infty$, we can still study the this sum formally, truncating it according to \cref{the:implicit imde}.
}
of $\Phi_h$ as $\hat{\fRHS}_h =\sum_{k=0}^{\infty} h^k\hat{\fRHS}_k$, and the corresponding IMDE via unrolled approximation $\Phi_h^{L}$ (\cref{eq:irk} or \cref{eq:irk_nr}) as $\fRHS{}_h =\sum_{k=0}^{\infty} h^k \fRHS{}_k$, respectively. Then
\begin{equation*}
\hat{\fRHS}_h - \fRHS{}_{h} = \mathcal{O}(h^{L^{*}+1}),\ \text{i.e.},  \hat{\fRHS}_k = \fRHS{}_{k} \text{ for } k= 0,\cdots, L^{*}, 
\end{equation*}
where $L^{*}=L$ for the unrolled approximation using fixed-point iteration \cref{eq:irk} and $L^{*}=2^{L+1}-2$ for the unrolled approximation using Newton-Raphson iteration \cref{eq:irk_nr}.
\end{namedTheorem}
\begin{proof}
The proof can be found in \cref{app:the:hat}.
\end{proof}
Additionally, with the tools of IMDEs, we can obtain the order of convergence for learning ODEs with unrolled implicit integration:
\begin{namedTheorem}[\textbf{Order of convergence for learning ODEs}]\label{the:error}
With the notation and under the conditions of \cref{the:implicit imde} and \cref{the:hat}, if $\Phi_h$ is of order $p$, i.e., $\Phi_{h}(\vec{x})=\phi_{h}(\vec{x})+\mathcal{O}(h^{p+1})$, and $L^{*}+1\geq p$, then, 
\begin{equation*}
\norm{\fRHS{}_{\theta}(\vec{x}) - \fRHS{}(\vec{x})}\leq c_2mh^p + \frac{e}{e-1}\mathcal{L},
\end{equation*}
where the constant $c_2$ depends only on $m/r_1$, $m/r_2$, $s$, $\Phi_h$ and the implicit solver.
\end{namedTheorem}
\begin{proof}
The proof can be found in \cref{app:the:error}.
\end{proof}

\section{Implementation of implicit scheme}

As discussed in \cref{sec: problem setup}, one has to exploit a root-finding algorithm to solve the implicit equation \cref{eq:rk1} for an implementation of \cref{eq:rk}.
However, a drawback is that the iteration number, or stopping criterion, should usually be determined in advance and fixed during training.
According to \cref{the:implicit imde} and \cref{the:hat},
different iteration numbers lead to different approximation targets
and increasing the iteration number results in a more accurate target.
Therefore, our goal is to provide {\em an adaptive algorithm} that increases the iteration number $L$,
such that the error of the unrolled approximation is less than the current learning loss, 
thereby increasing computational efficiency while preserving accuracy.

\begin{algorithm}
\caption{\textbf{Training with adaptive iteration}}
\label{Training with adaptive iteration}
\begin{algorithmic}[1]
\STATE{Initialization: $L$, and a neural network $\fRHS{}_{\theta}$ with trainable parameters.}
\FOR{each training epoch}

    \STATE{Compute $\text{Loss} =\frac{1}{D \cdot N\cdot M}\mathcal{L}_{unrolled}$, where $D$ is the dimension.
    \label{algstep:computeLoss}}
    \STATE{Let $\theta \leftarrow \text{optimizer}(
                \theta,
                \text{lr},
                \frac{\partial \text{Loss}}{\partial \theta}
            )$
    to update neural network parameters, where $\text{lr}$ is the learning rate.
    \label{algstep:paramUpdate}
    }
    \IF{adjust iteration number \label{algstep:adjustIterNum}}
    \STATE{
        $\delta = \frac{1}{D \cdot N\cdot M}\sum_{n=1}^N\sum_{m=1}^{M} \norm{\big(\Phi_{h,\fRHS{}_{\theta}}^{L+1}\big)^{ms}(\vec{x}_n) - \big(\Phi_{h,\fRHS{}_{\theta}}^L\big)^{ms}(\vec{x}_n)}_2^2/(m\Delta t)^2$.
        \label{algstep:computeDelta}
    }
    \IF{$\text{Loss} < c \delta$ }
    \STATE{Increase the iteration number $L$.}
    \ENDIF
    \ENDIF

    \ENDFOR
\end{algorithmic}
\end{algorithm}

Next, we present the error quantification for an ODE-net using an unrolled implicit scheme, which will form the cornerstone for the following adaptive algorithm.

\begin{namedLemma}[\textbf{Convergence of the (``inner") implicit iteration}]\label{lem:fp}
Consider a consistent implicit Runge-Kutta scheme $\Phi_h$ \cref{eq:rk} and its approximation
$\Phi_h^{L}$
using fixed-point iteration
\cref{eq:irk} or Newton-Raphson iteration \cref{eq:irk_nr}. {Then, 
\begin{equation*}
\begin{aligned}
\mathcal{L}_{exact}^{\frac{1}{2}} :=&\left(\sum_{n=1}^N\sum_{m=1}^{M} \norm{\brac{\Phi_{h,\fRHS{}_{\theta}}}^{ms}(\vec{x}_n)- \phi_{m \Delta t}(\vec{x}_n)}_2^2/(m\Delta t)^2\right)^{\frac{1}{2}}\\
\leq & \mathcal{L}_{unrolled}^{\frac{1}{2}}+ \left(\sum_{n=1}^N\sum_{m=1}^{M} \norm{\big(\Phi_{h,\fRHS{}_{\theta}}^{L+1}\big)^{ms}(\vec{x}_n) - \big(\Phi_{h,\fRHS{}_{\theta}}^L\big)^{ms}(\vec{x}_n)}_2^2/(m\Delta t)^2 \right)^{\frac{1}{2}}+\mathcal{O}(h^{(L+1)^{*}+1}),
\end{aligned}
\end{equation*}}
where $L^{*}=L$ for the unrolled approximation using fixed-point iteration \cref{eq:irk} and $L^{*}=2^{L+1}-2$ for the unrolled approximation using Newton-Raphson iteration \cref{eq:irk_nr}.
\end{namedLemma}
\begin{proof}
The proof can be found in \cref{app:lem:fp}.
\end{proof}

According to this inequality, we formulate our adaptive  \cref{Training with adaptive iteration}.
The core idea is to monitor the level of error, and adapt the iteration number in the training process according to \cref{lem:fp}. Essentially, \cref{Training with adaptive iteration} adjusts the approximation target, i.e., the IMDE of $\Phi_h^L$, to gradually approach the true target, i.e., the IMDE of $\Phi_h$, see \cref{fig:ill} for an illustration.

\begin{figure}[htb]
    \centering
    \includegraphics[width=0.7\linewidth]{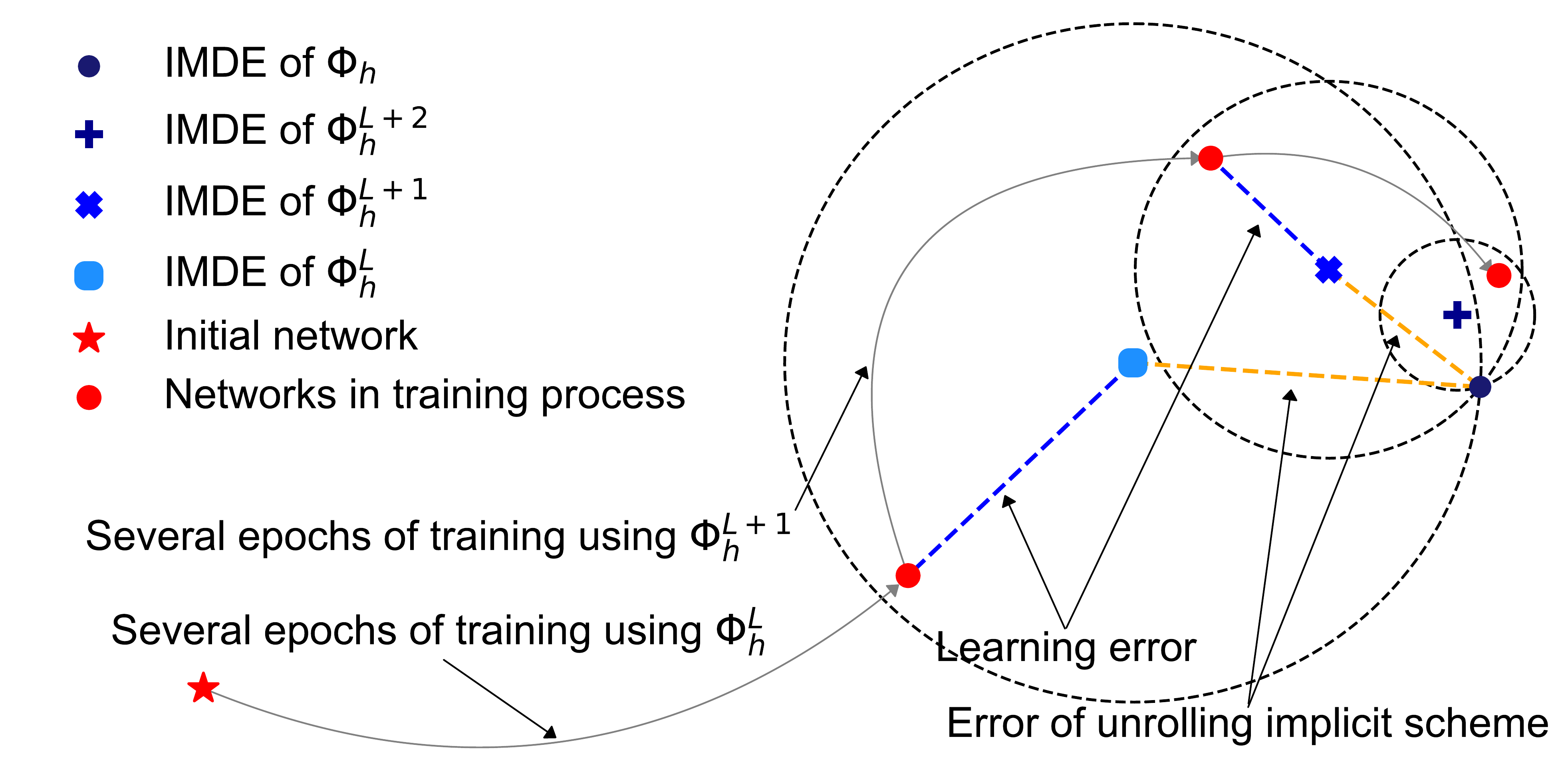}
    \caption{
        \textbf{Illustration of the proposed adaptive algorithm.} We initially employ a smaller iteration number $L$ to train the neural network, and we gradually increase $L$ as the learning error decreases for more precise approximation target.}
    \label{fig:ill}
\end{figure}

\section{Numerical examples}\label{sec:Numerical examples}
In this section, several examples are used to demonstrate the performance of the proposed algorithm and verify the theoretical analysis.
We use the PyTorch library
to implement \cref{Training with adaptive iteration} to train our neural networks.
For a given implicit solver (e.g., fixed-point iteration or Newton-Raphson iteration), we can store and backpropagate through all the iterations to obtain exact gradients for optimization.
For all experiments except the last one, we generate the state data by numerically solving the dynamical system using a high order integrator with a tiny adaptive step.
The trajectories are ``split" so that their length $M$ is 1. The last of our experiments uses real-world data~\cite{schmidt2009distilling}.
In this last case, due to the measurement errors and other non-ideal effects, we set the length of divided trajectories to $M=10$.
After training, we simulate the learned system using a high resolution numerical solver and compare it against the true system solution. 
Specifically, the numerical solver for generating data and solving learned system is the fourth-order Runge-Kutta method with a finer time step of size $0.01\cdot\Delta t$.

\subsection{Linear ODEs}\label{sec:Linear ODEs}
We first present some numerical results for  two-dimensional {\em linear} ODEs, to verify that training an ODE-net using an unrolled implicit scheme returns an approximation of the IMDE. All examples are taken from~\cite{wu2019numerical}.

\begin{table}[!htb]
 \small
  \centering
  \resizebox{\textwidth}{!}{
  \begin{tabular}{p{0.23\textwidth}p{0.19\textwidth}p{0.18\textwidth}p{0.18\textwidth}p{0.22\textwidth}}
    \hline
    Phase portrait
    &\makecell[l]{True system\\     \begin{minipage}[b]{0.1\textwidth}
		\centering
		\raisebox{-.1\height}{\includegraphics[width=\linewidth]{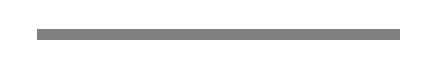}}
	\end{minipage} }
	&\makecell[l]{Learned system \\     \begin{minipage}[b]{0.1\textwidth}
		\centering
		\raisebox{-.1\height}{\includegraphics[width=\linewidth]{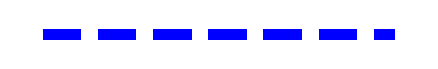}}
	\end{minipage} }
	&\makecell[l]{IMDE\\     \begin{minipage}[b]{0.1\textwidth}
		\centering
		\raisebox{-.1\height}{\includegraphics[width=\linewidth]{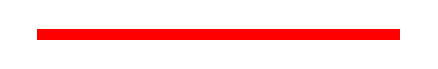}}
	\end{minipage} }
	& Settings
    \\ \hline
    \begin{minipage}[b]{0.2\textwidth}
		\centering
		\raisebox{-.5\height}{\includegraphics[width=\linewidth]{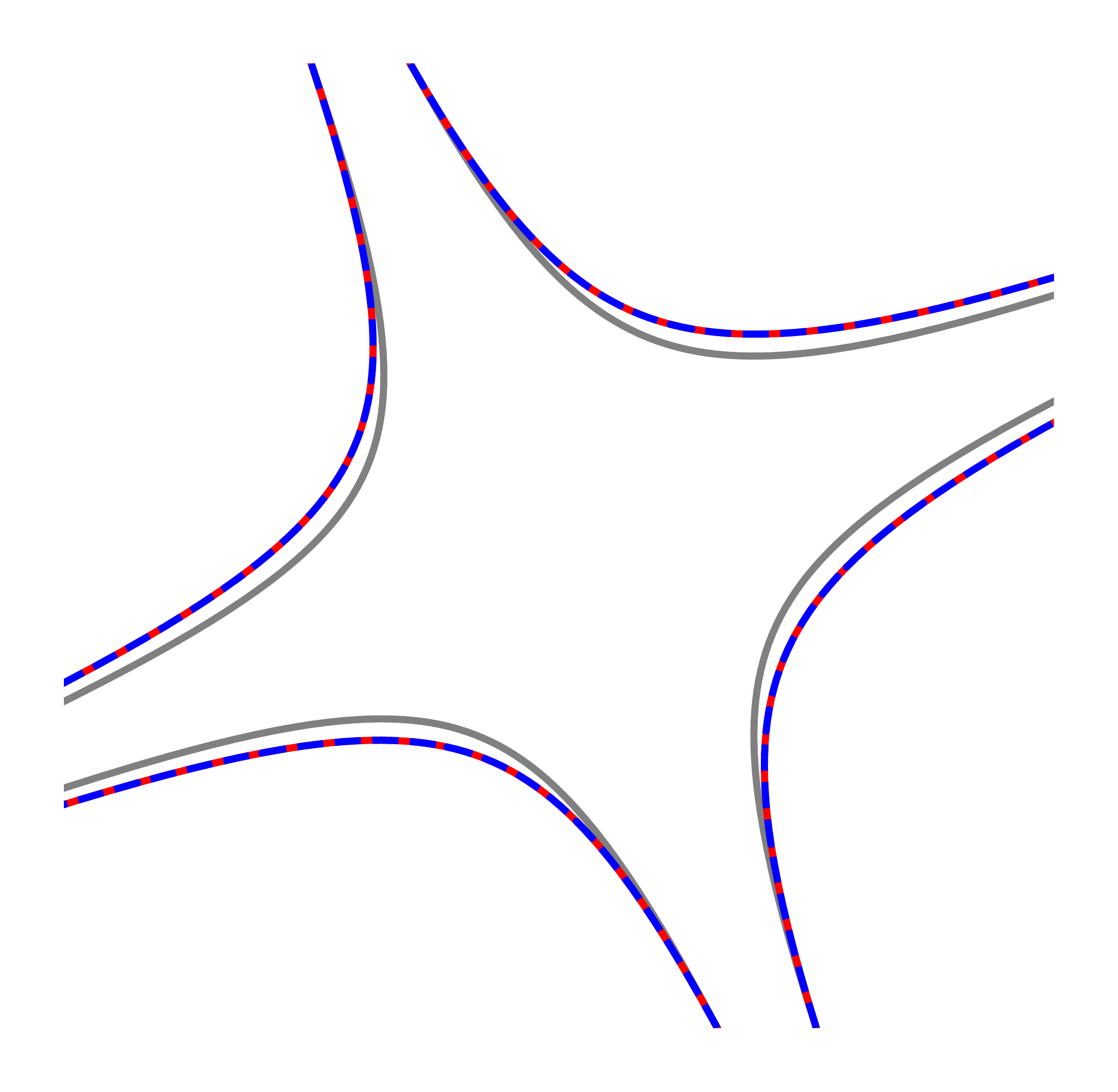}}
	\end{minipage}
    &\makecell[l]{Saddle point\\$\begin{aligned}\frac{d}{dt}p =& p+q-2\\\frac{d}{dt}q =& p-q\end{aligned}$}
    &{$\begin{aligned}\frac{d}{dt}p =&1.1035p\\ &+1.0033q\\ &-2.1068\\ \frac{d}{dt}q = &1.0033p\\ &-0.9032q\\ &-0.1002\end{aligned}$}
    &{$\begin{aligned}\frac{d}{dt}p =&1.1035p\\ &+1.0033q\\ &-2.1068\\ \frac{d}{dt}q = &1.0033p\\ &-0.9032q\\ &-0.1002\end{aligned}$}
    & \makecell[l]{$\mathcal{D} =[0,2]^2$\\ $\Delta t=0.1$\\$s=1$\\$L=0$\\ \\
    All schemes are \\equivalent in this\\ case.}
    \\ \hline
    \begin{minipage}[b]{0.2\textwidth}
		\centering
		\raisebox{-.5\height}{\includegraphics[width=\linewidth]{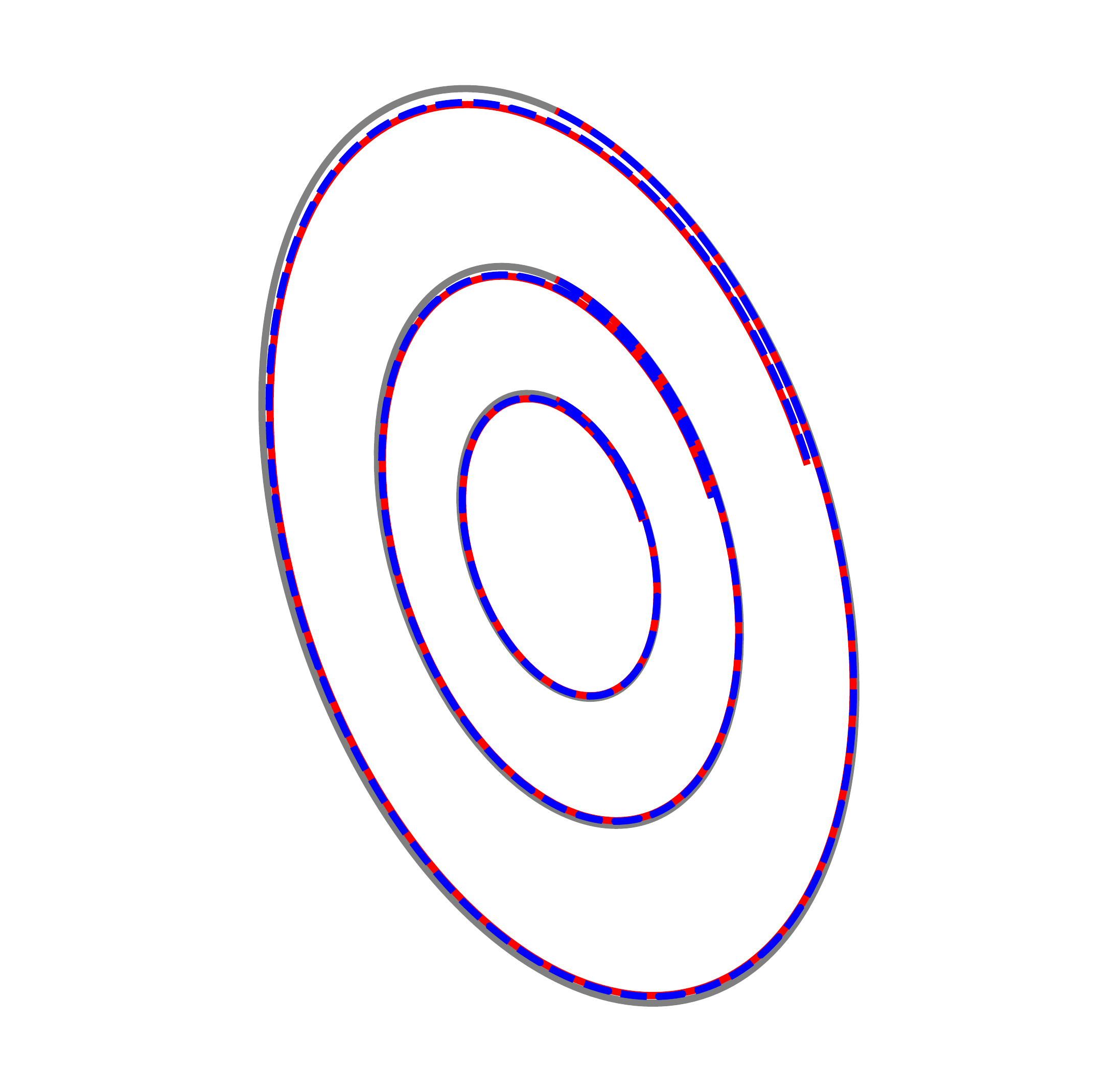}}
	\end{minipage}
    &\makecell[l]{Center point\\$\begin{aligned}\frac{d}{dt}p =& p+2q\\\frac{d}{dt}q =& -5p-q\end{aligned}$}
    &{$\begin{aligned}\frac{d}{dt}p =&0.9651p\\ &+1.9607q\\ &+0.0000\\ \frac{d}{dt}q = &-4.9017p\\ &-0.9956q\\ &+0.0000\end{aligned}$}
    &{$\begin{aligned}\frac{d}{dt}p =&0.9609p\\ &+1.9568q\\ &+0.0000\\ \frac{d}{dt}q = &-4.8920p\\ &-0.9959q\\ &+0.0000\end{aligned}$}
    & \makecell[l]{$\mathcal{D} =[-1,1]^2$\\ $\Delta t=0.12$\\$s=1$\\ \\ Implicit Trapezoidal \\ using fixed-point \\iteration with $L=1$}
    \\\hline
    \begin{minipage}[b]{0.2\textwidth}
		\centering
		\raisebox{-.5\height}{\includegraphics[width=\linewidth]{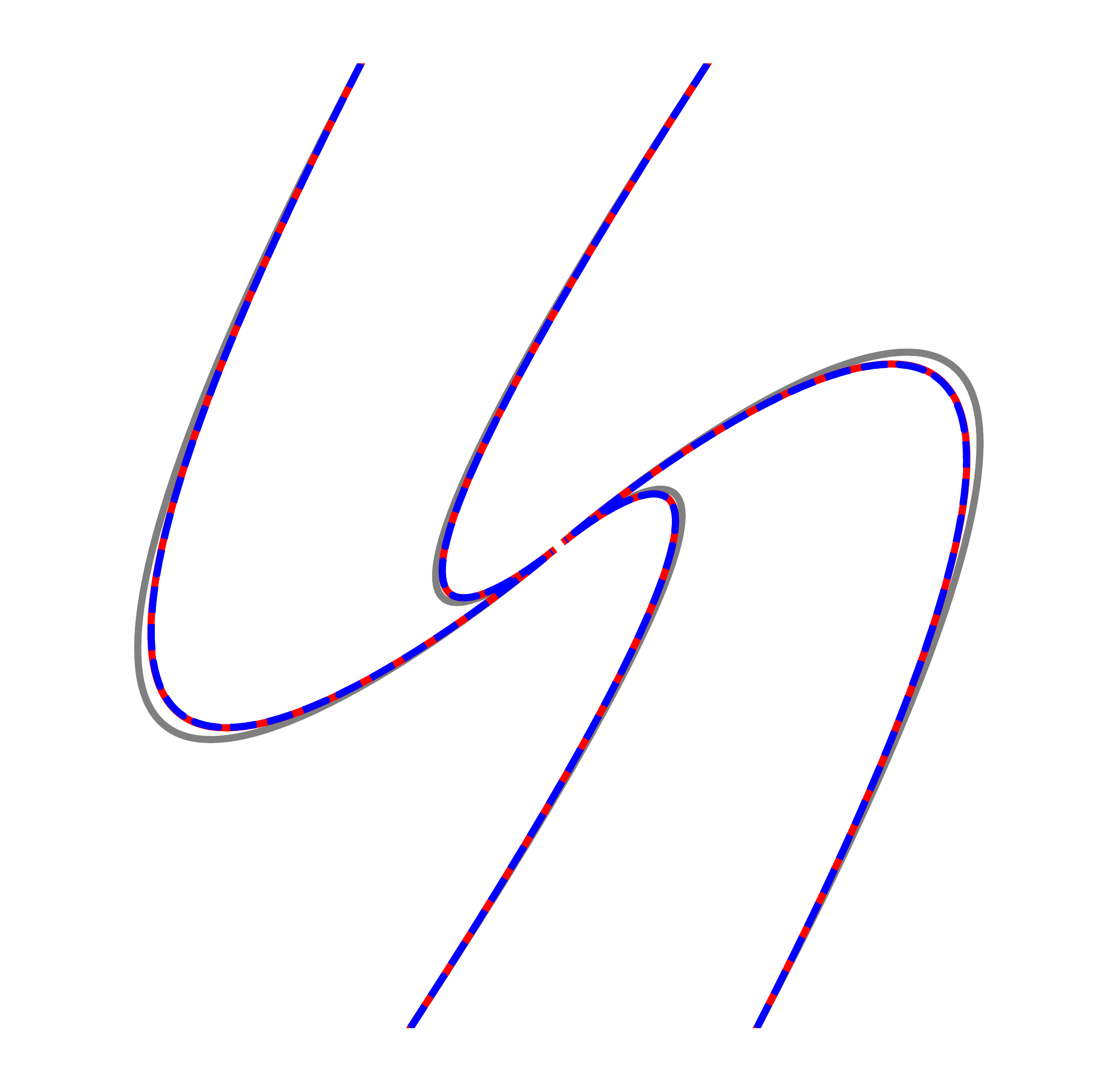}}
	\end{minipage}
    &\makecell[l]{Improper node\\{$\begin{aligned}\frac{d}{dt}p =&p-4q\\ \frac{d}{dt}q = &4p-7q\end{aligned}$}}
    &{$\begin{aligned}\frac{d}{dt}p =&0.8222p\\ &-3.7709q\\ &+0.0000\\ \frac{d}{dt}q = &3.7709p\\ &-6.7197q\\ &+0.0000\end{aligned}$}
    &{$\begin{aligned}\frac{d}{dt}p =&0.8270p\\ &-3.7771q\\ &+0.0000\\ \frac{d}{dt}q = &3.7771p\\ &-6.7272q\\ &+0.0000\end{aligned}$}
    & \makecell[l]{$\mathcal{D} =[-1,1]^2$\\ $\Delta t=0.12$\\$s=1$\\ \\ Implicit Midpoint \\using fixed-point \\iteration with $L=2$}
    \\\hline
    \begin{minipage}[b]{0.2\textwidth}
		\centering
		\raisebox{-.5\height}{\includegraphics[width=\linewidth]{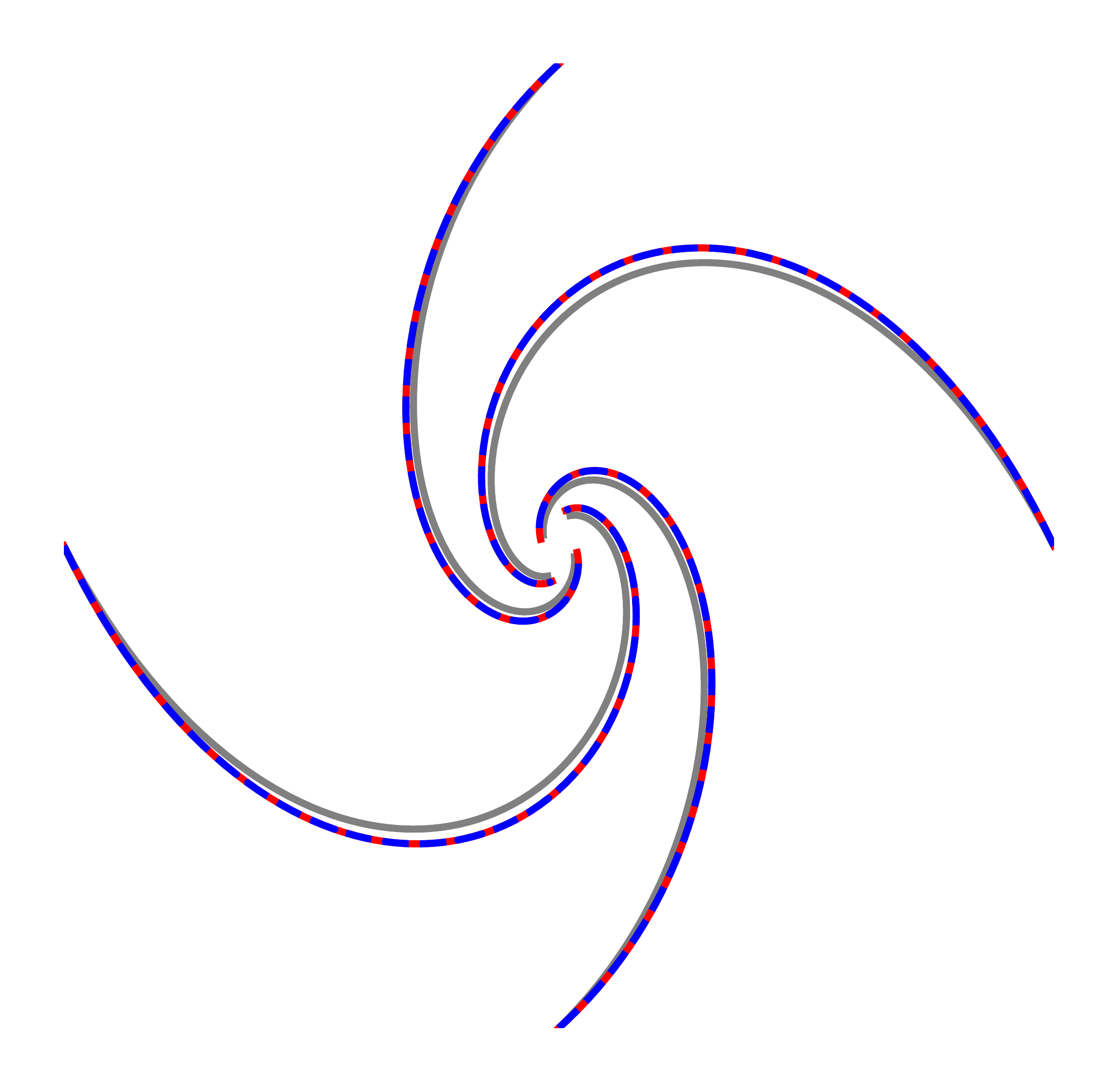}}
	\end{minipage}
    &\makecell[l]{Spiral point\\$\begin{aligned}\frac{d}{dt}p =&-p-q-1\\ \frac{d}{dt}q = &2p-q+5\end{aligned}$}
    &{$\begin{aligned}\frac{d}{dt}p =&-0.9729p\\ &-1.0503q\\ &-0.8955\\ \frac{d}{dt}q = &2.1006p\\ &-0.9729q\\ &+5.1742\end{aligned}$}
    &{$\begin{aligned}\frac{d}{dt}p =&-0.9729p\\ &-1.0504q\\ &-0.8954\\ \frac{d}{dt}q = &2.1008p\\ &-0.9729q\\ &+5.1745\end{aligned}$}
    & \makecell[l]{$\mathcal{D} =[-3,-1]\times[0,2]$\\ $\Delta t=0.05$\\$s=1$\\ \\ Implicit Euler using\\fixed-point iteration \\with $L=3$ }
    \\\hline
    \begin{minipage}[b]{0.2\textwidth}
		\centering
		\raisebox{-.5\height}{\includegraphics[width=\linewidth]{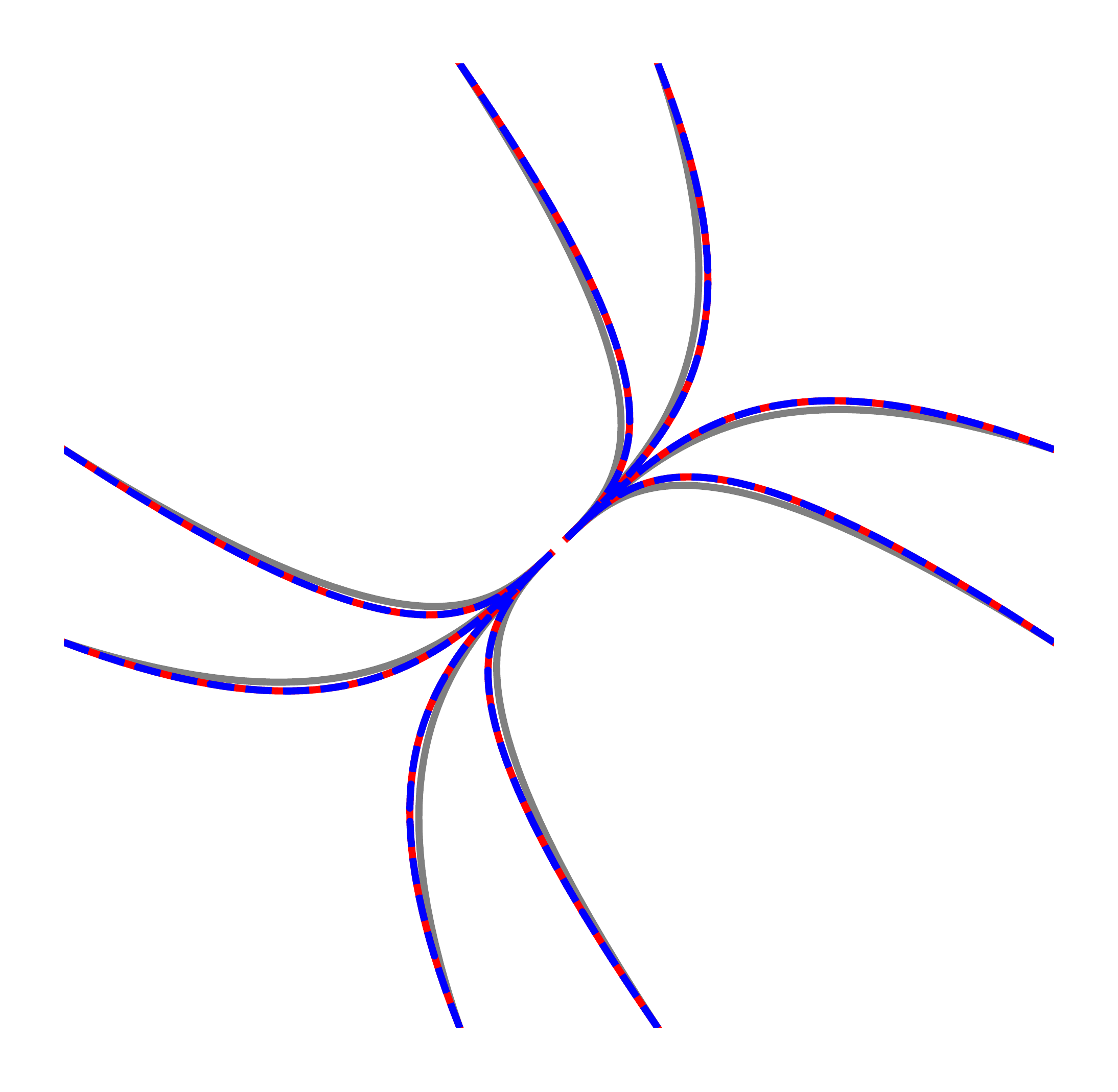}}
	\end{minipage}
    &\makecell[l]{Nodal sink\\$\begin{aligned}\frac{d}{dt}p =&-2+q-2\\ \frac{d}{dt}q = &p-2q+1\end{aligned}$}
    &{$\begin{aligned}\frac{d}{dt}p =&-2.3368p\\ &+1.2743q\\ &-2.3368\\ \frac{d}{dt}q = &1.2743p\\ &-2.3368q\\ &+1.2743\end{aligned}$}
    &{$\begin{aligned}\frac{d}{dt}p =&-2.3366p\\ &+1.2741q\\ &-2.3366\\ \frac{d}{dt}q = &1.2741p\\ &-2.3366q\\ &+1.2741\end{aligned}$}
    & \makecell[l]{$\mathcal{D} =[-2,0]\times[-1,1]$\\ $\Delta t=0.12$\\$s=1$\\ \\ Implicit Euler using\\Newton-Raphson \\ iteration with $L=1$ }
    \\ \hline
  \end{tabular}}
  \caption{
    \textbf{Discovery of linear systems.}
    The leftmost column shows the phase portraits of the true, learned and modified systems.
    The four columns on the right give the corresponding
    ODEs as well as the experiment details.
    Here, $L$ is the fixed iteration number,
    and $L=0$ means that there is no iteration;
    i.e., the scheme reduces to forward Euler
    $\Phi_{h,\fRHS{}}^0(\vec{x}) = \vec{x} + h \fRHS{}(\vec{x})$. }
  \label{tab:lin}
\end{table}

For each test in this subsection, the training data is composed of $100$ points 
    generated from a uniform distribution over a computational domain $\mathcal{D}$;
    each paired with its time-$\Delta t$ flow;
    i.e., $\{\vec{x}_n, \phi_{\Delta t}(\vec{x}_n)\}_{n=1}^{100}$, with 
    $\vec{x}_n \sim \text{Uniform}(\mathcal{D})$.
The ODE solver that used to learn the ODE
was chosen to be a single composition ($s=1$) of a chosen unrolled implicit scheme with a fixed iteration number $L$.
For the linear case, the employed neural networks have a single linear layer, i.e., we learn an affine transformation $ \fRHS{}_{\theta}(\vec{x})= \mat{W}\vec{x}+\vec{b}$, where $\mat{W}\in \R^{D\times D}, \vec{b}\in\R^D$ are the $D^2+D$ learnable parameters. 
We use full-batch Adam optimization~\cite{kingma2014adam} with a learning rate of $0.01$ to update the parameters $10^4$ times.

The detailed computational settings, descriptions of the systems 
and the corresponding numerical and analysis results are presented in \cref{tab:lin}.
Note
that Newton-Raphson iteration with $L=1$ can exactly solve the implicit linear equation
and thus higher iterations are not discussed.
As shown in the phase portraits,
the ODE-nets accurately capture the evolution of the corresponding IMDE.
In addition, the trajectories of the learned systems 
are closer to those of the Modified systems than to the those of the true ODE,
which confirms that training an ODE-net returns an approximation of the IMDE.

\subsection{Damped pendulum problem}\label{sec:dpb}
We now consider the damped pendulum problem,
\begin{equation*}
\begin{aligned}
\frac{d}{dt}p =&-\alpha p - \beta \sin q,\\
\frac{d}{dt}q =&p,\\
\end{aligned}
\end{equation*}
where $\alpha = 0.2$ and $\beta = 8.91$.
\begin{figure}[htb]
    \centering
    \includegraphics[width=0.8\linewidth]{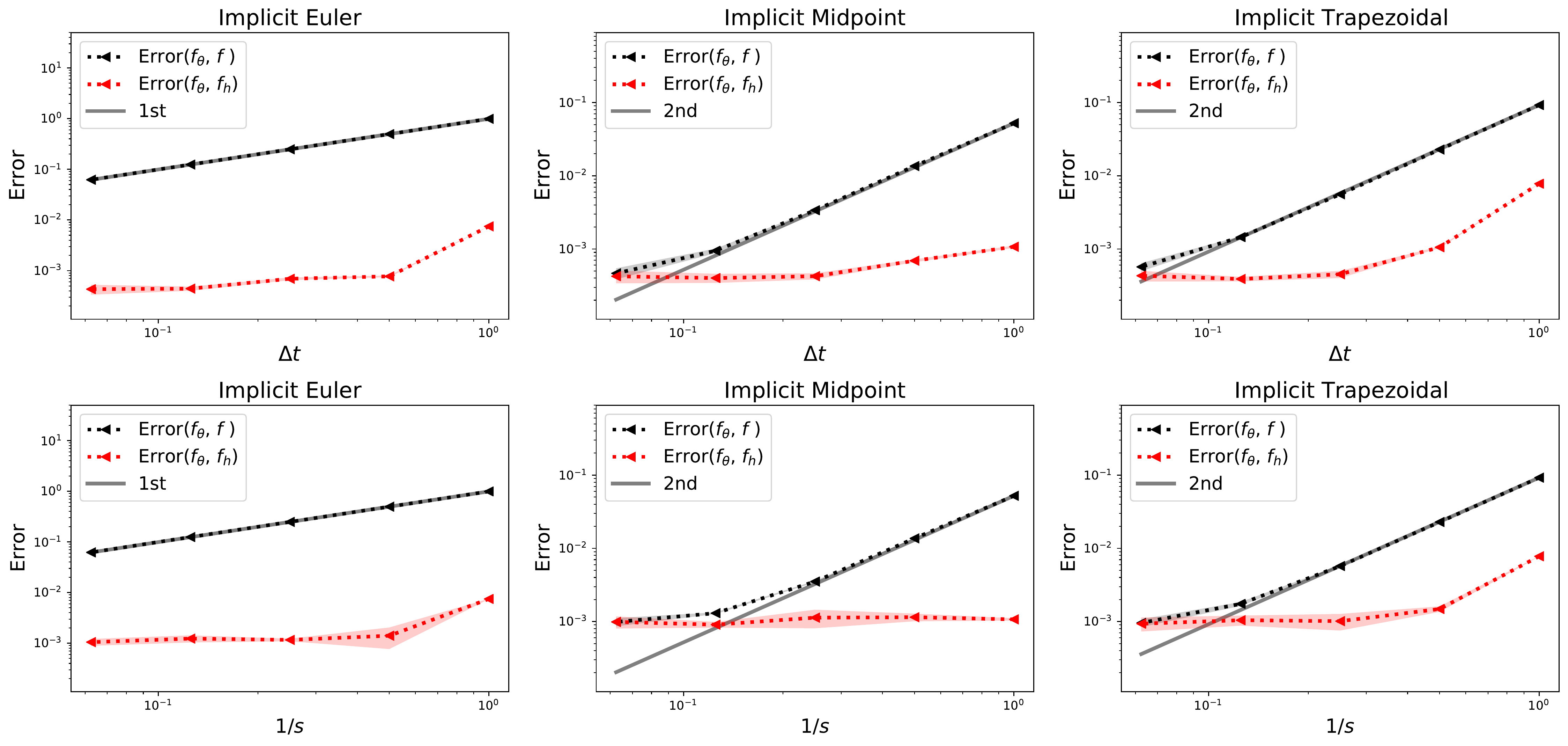}
    \caption{
    \textbf{Convergence rate test with respect to $h=\Delta t/s$ for learning the damped pendulum system.}
    Here, $\Delta t$ is the data step size,
    $h$ is the numerical scheme step size,
    and therefore $s=\Delta t/h$ is the number of scheme compositions used.
    Since error $\sim h^p = (\Delta t/s)^p$, we set $1/s$ and $\Delta t$ as the horizontal coordinates to show convergence with respect to $h$.
    As we sweep $\Delta t$,
    we keep $h$ (i.e. $s$) constant, and vice-versa.
    We see that the $\text{Error}(\fRHS{}_{\theta}, \fRHS{})$ is more than the $\text{Error}(\fRHS{}_{\theta}, \fRHS{}_h)$. 
    The order of $\text{Error}(\fRHS{}_{\theta}, \fRHS{})$ with respect to $h$ is consistent with the order of the employed numerical schemes. The results are obtained by taking the mean of $5$ independent experiments, and the shaded
    region represents one standard deviation.
    }
    \label{fig:DPerror}
\end{figure}
\noindent
We generate 90 and 10 trajectories from $t=0$ to $t=4$ for the training data and test data respectively,
with initial points randomly sampled from a uniform distribution on $\mathcal{D}=[-1.5,0]\times [-4,0]$.
For each trajectory, $4/\Delta t +1$ data points at equidistant time steps $\Delta t$ are selected and grouped in $4/\Delta t$ successive $M=1$ pairs.
Here we employ fixed-point iteration
to solve the implicit equation.
We use a feedforward neural network with two hidden layers to represent the unknown vector field, i.e.,
\begin{equation*}
\fRHS{}_{\theta}(\vec{x}) = \mat{W}_3 \texttt{tanh}(\mat{W}_2 \texttt{tanh} (\mat{W}_1 \vec{x}+\vec{b}_1)+\vec{b}_2)+\vec{b}_3,
\end{equation*}
where $\mat{W}_1\in \R^{128\times D},\ \mat{W}_2\in \R^{128\times 128},\ \ \mat{W}_3\in \R^{D\times 128},\ \vec{b}_1, \vec{b}_2\in \R^{128},\ \vec{b}_3\in\R^D$ are the $256D+128^2+256+D$ learnable parameters and $D=2$ is the state dimension. Results are collected after $10^5$ parameter updates using full-batch Adam optimization; the learning rate is set to decay exponentially with, linearly decreasing power from $10^{-2}$ to $10^{-4}$. We also include comparisons with the fixed iteration number setting, where we apply $L=5$ iterations.

\begin{figure}[!ht]
    \centering
    \includegraphics[width=\linewidth]{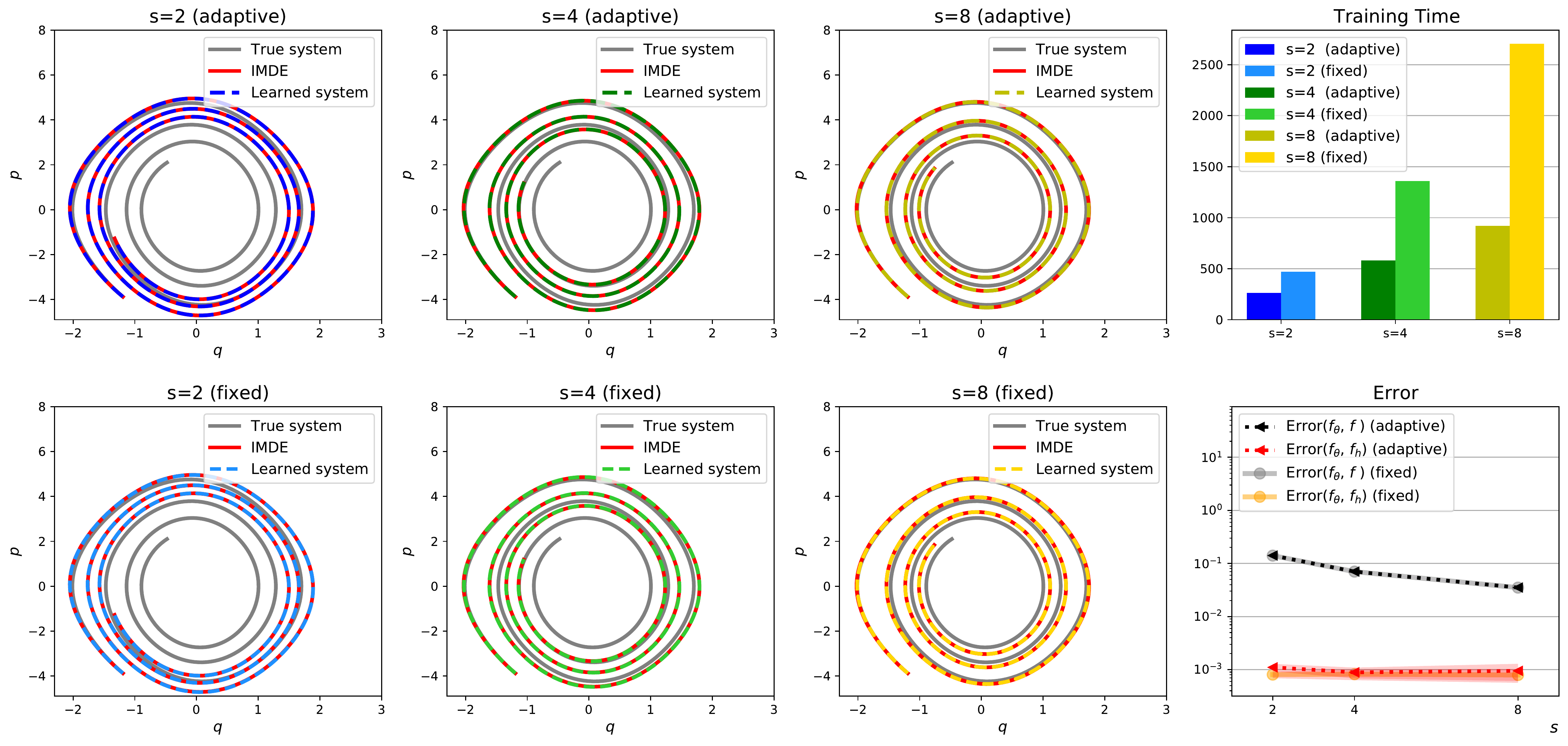}
    \caption{
        \textbf{Results for learning damped pendulum problem.}
        The integration of the learned system always matches that of the derived IMDE more than the truth, but the adaptive algorithm extracts this IMDE more quickly.
    }
    \label{fig:DP}
\end{figure}

We first verify the convergence rate with respect to the step size $h$. Here, we evaluate the average error between $\fRHS{}_{\theta}$ and $\fRHS{}$ and between $\fRHS{}_{\theta}$ and $\fRHS{}_h$ in the $l_{\infty}$- norm, i.e., 
\begin{equation}\label{eq:error of experiments}
\begin{aligned}
&\text{Error}(\fRHS{}_{\theta}, \fRHS{}) = \frac{1}{|\mathcal{T}_{test}|}\sum_{\vec{x}_n\in \mathcal{T}_{test}}  \norm{\fRHS{}_{\theta}(\vec{x}_n) - \fRHS{}(\vec{x}_n)}, \\ 
&\text{Error}(\fRHS{}_{\theta}, \fRHS{}_h) = \frac{1}{|\mathcal{T}_{test}|}\sum_{\vec{x}_n\in \mathcal{T}_{test}}  \norm{\fRHS{}_{\theta}(\vec{x}_n) - \fRHS{}_h(\vec{x}_n)}, 
\end{aligned}
\end{equation}
where $\mathcal{T}_{test} = \{\phi_{m\Delta t}(\vec{x}_n)\}_{1\leq n \leq 10,\ 0\leq m \leq 4/\Delta t, \Delta t =0.01}$ is the total test data when $ \Delta t =0.01$.
We assign various step sizes $\Delta t$ as $\Delta t=0.01\cdot2^k, \ k=0, \cdots, 4$ 
with fixed composition number $s=1$; we also use 
several composition numbers $s=2^k, \ k=0, \cdots, 4$
with fixed horizon $\Delta t =0.16$, respectively.
The errors are recorded in \cref{fig:DPerror}.
It can be seen that the $\text{Error}(\fRHS{}_{\theta}, \fRHS{})$ is markedly more than the $\text{Error}(\fRHS{}_{\theta}, \fRHS{}_h)$, 
indicating that the learned ODE-net returns an approximation of the particular IMDE rather than the true ODE.
In addition, the order of $\text{Error}(\fRHS{}_{\theta}, \fRHS{})$ with respect to $h$
is consistent with the order of the employed numerical schemes when $h$ is relatively large,
since the learning error dominates the overall error for an accurate solver.

Next, we simulate the exact solution from $t=0$ to $t=8$ using the initial condition $y_0=(-3.876,-1.193)$.
We show in \cref{fig:DP} the exact trajectories of the true system, the corresponding IMDE, and the right-hand-sides learned by ODE-net for different schemes, where $\Delta t=0.01, s=2, 4, 8$.
For all integrations, the ODE-net accurately captures the evolution of the corresponding IMDE, which again implies that the learned ODE-net returns an approximation of the IMDE.
When using a small learning time step $h=\Delta t/s$, the difference between the IMDE and the original equation is reduced, and thus the ODE-net tends to learn the true system.
In addition, we record the error and the training time on the right side of \cref{fig:DP}.
It is observed that the proposed adaptive iteration algorithm is remarkably faster than the non-adaptive, direct implementation, and requires less training wall-clock time to reach similar accuracy.

\subsection{Glycolytic oscillator}
As an example of an initial value solver employing the Newton-Raphson iteration, we consider a model of oscillations in yeast glycolysis \cite{daniels2014efficient}.
The model describes the concentrations of seven biochemical species and is defined by
\begin{equation*}
\begin{aligned}
\frac{d}{dt}S_1 =& J_0 - \frac{k_1 S_1S_6}{1+(S_6/K_1)^q},\\
\frac{d}{dt}S_2 =& 2 \frac{k_1 S_1S_6}{1+(S_6/K_1)^q} - k_2S_2(N-S_5)-k_6S-2S_5,\\
\frac{d}{dt}S_3 =& k_2S_2(N-S_5) - k_3S_3(A-S_6),\\
\frac{d}{dt}S_4 =& k_3S_3(A-S_6) - k_4S_4S_5 - \kappa(S_4-S_7),\\
\frac{d}{dt}S_5 =& k_2S_2(N-S_5) - k_4S_4S_5 - k_6S_2S_5,\\
\frac{d}{dt}S_6 =& -2 \frac{k_1 S_1S_6}{1+(S_6/K_1)^q}  +2k_3S_3(A-S_6)-k_5S_6,\\
\frac{d}{dt}S_7 =&\psi \kappa (S_4-S_7) - kS_7,
\end{aligned}
\end{equation*}
where the ground truth parameters are taken from Table 1 in
~\cite{daniels2014efficient}.

\begin{table}[htbp]
    \centering
    \resizebox{\textwidth}{!}{
    \begin{tabular}{l|cccc}
    \hline
       Methods & \makecell[c]{Implicit Midpoint \\(adaptive)}
       & \makecell[c]{Implicit Midpoint \\(fixed)}
       & \makecell[c]{Implicit Trapezoidal \\(adaptive) }
       & \makecell[c]{Implicit Trapezoidal \\(fixed) }\\
     \hline
       Training time & $1763\pm 87$ & $5879 \pm 224$ & $2264\pm 92$ & $6357\pm 167$\\
    \hline
        Error($\fRHS{}_{\theta}$, $\fRHS{}$) & 2.91e-2 $\pm$ 2.37e-3 & 2.91e-2 $\pm$ 2.76e-3 & 4.05e-2 $\pm$ 3.14e-3 & 4.09e-2 $\pm$ 3.21e-3\\
     \hline
       Error($\phi_{T, \fRHS{}_{\theta}}$, $\phi_{T, \fRHS{}}$) & 7.62e-3 $\pm$ 4.10e-3 & 7.63e-3 $\pm$ 3.70e-3 & 1.19e-2 $\pm$ 5.30e-3 & 8.74e-3 $\pm$ 5.46e-3\\
     \hline
    \end{tabular}
    }
    \caption{
    \textbf{The training time (in seconds) and global error for learning the glycolytic oscillator.}
    The results are recorded in the form of mean $\pm$ standard deviation based on 10 independent training.
    The proposed adaptive algorithm markedly decrease training time with no compromise in accuracy.
    }
    \label{tab:GO}
\end{table}

\begin{figure}[ht]
    \centering
    \includegraphics[width=\linewidth]{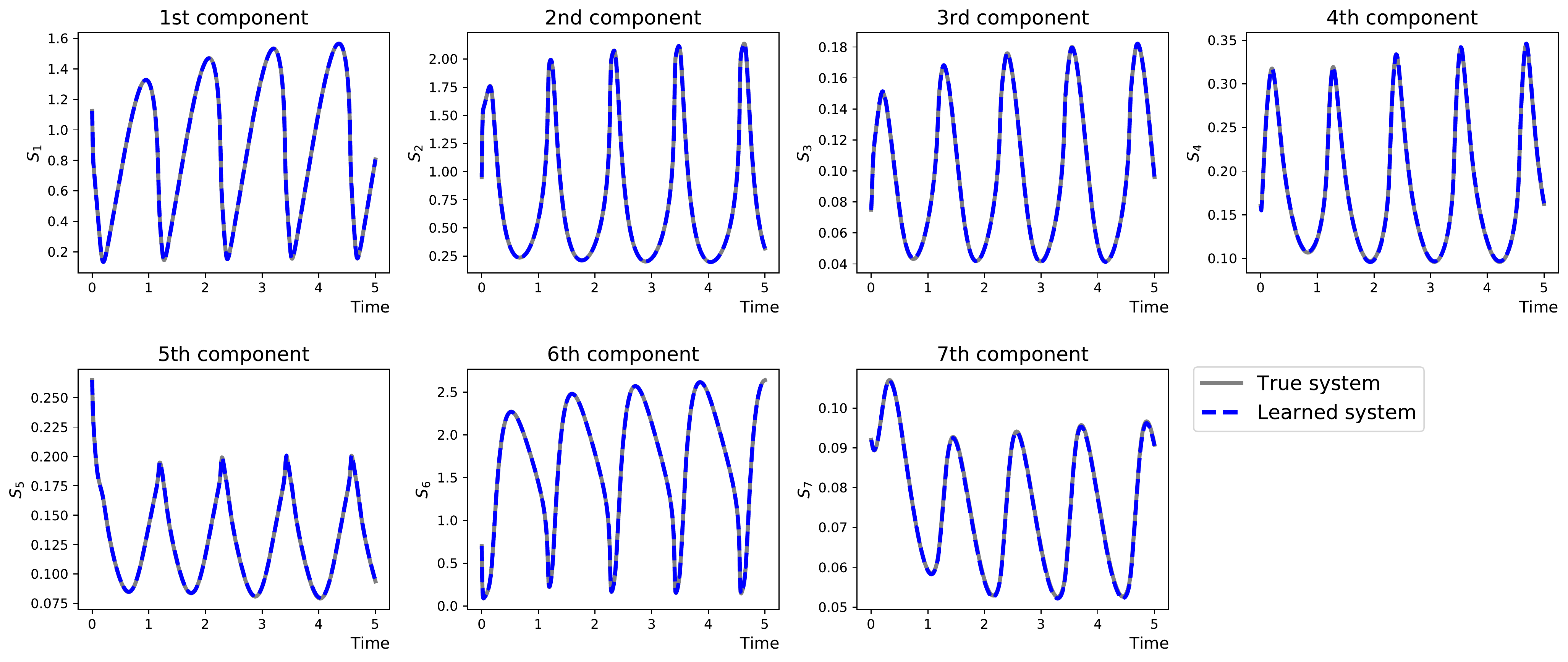}
    \caption{
        \textbf{Exact and learned dynamics of glycolytic oscillator.}
            For this evaluation experiment,
            we keep $\Delta t = 0.01$ as in training,
            but set $M=500$ rather than $M=1$
            to show that the long-term dynamics are accurate.
            Additionally, the initial condition used is not itself present in the training dataset.
      }
    \label{fig:GO}
\end{figure}

In this example, training data consists of $20$ simulations which start at $(1+\delta)\cdot \vec{x}_0$, where $\vec{x}_0=(1.125,0.95,0.075,0.16,0.265,0.7,0.092)^{\top} = (S_1, \ldots, S_7)^{\top}$ and $\delta$ is uniformly sampled from $[-0.2, 0.2]$.
On each trajectory, $500$ pairs of snapshots at $(i \Delta t, (i+1)\Delta t)$, $i = 0, \cdots, 499$, $\Delta t = 0.01$
are used as training data.
While the Newton-Raphson iteration is used, and $s$ is fixed to $2$, the chosen model architecture and hyperparameters are the same as in \cref{sec:dpb}.

After training, we record the training time and the error in \cref{tab:GO}; the error between $\fRHS{}_{\theta}$ and $\fRHS{}$ is evaluated via \cref{eq:error of experiments} with $\mathcal{T}_{test} = \{\phi_{m\Delta t}(\vec{x}_0)\}_{\ 0\leq m \leq 500}$, 
while the error between the learned and exact trajectories are evaluated by
\begin{equation*}
\text{Error}(\phi_{T, \fRHS{}_{\theta}}, \phi_{T, \fRHS{}}) = \frac{\Delta t}{T}\sum_{m=1}^{T/\Delta t}  \norm{\phi_{m\Delta t, \fRHS{}_{\theta}}(\vec{x}_0) - \phi_{m\Delta t, \fRHS{}}(\vec{x}_0)}.
\end{equation*}
As can be seen from \cref{tab:GO}, the proposed algorithm leads to a $2$-$3\times$ speedup in training without noticeable degradation in accuracy.

In addition, we use an implicit midpoint scheme to learn
the system from initial condition $\vec{x}_0$
and depict the learned and exact dynamics in \cref{fig:GO}.
It can be seen that the system learned using the proposed algorithm correctly captures the form of the dynamics, indicating that the performance of our approach is still promising for moderately high-dimensional equation discovery.

\subsection{Learning real-world dynamics}\label{sec: Learning real-world dynamics}
Finally, we use real-world data~\cite{schmidt2009distilling} to verify that the proposed algorithm can learn accurate dynamics and predict future behavior in real-world problems.
This data consists of about 500 points a single trajectory of two coupled  oscillators.
We use the first 3/4 of the trajectory for training, and the remainder for testing.
Here we assign $s=1$, and set the length of divided trajectories to $M=10$ rather than $1$ due to measurement errors and other non-ideal effects.
We train models by using full-batch 
Adam
with a learning rate of $10^{-4}$ over 5000 epochs.

\begin{figure}[htb]
    \centering
    \includegraphics[width=0.85\linewidth]{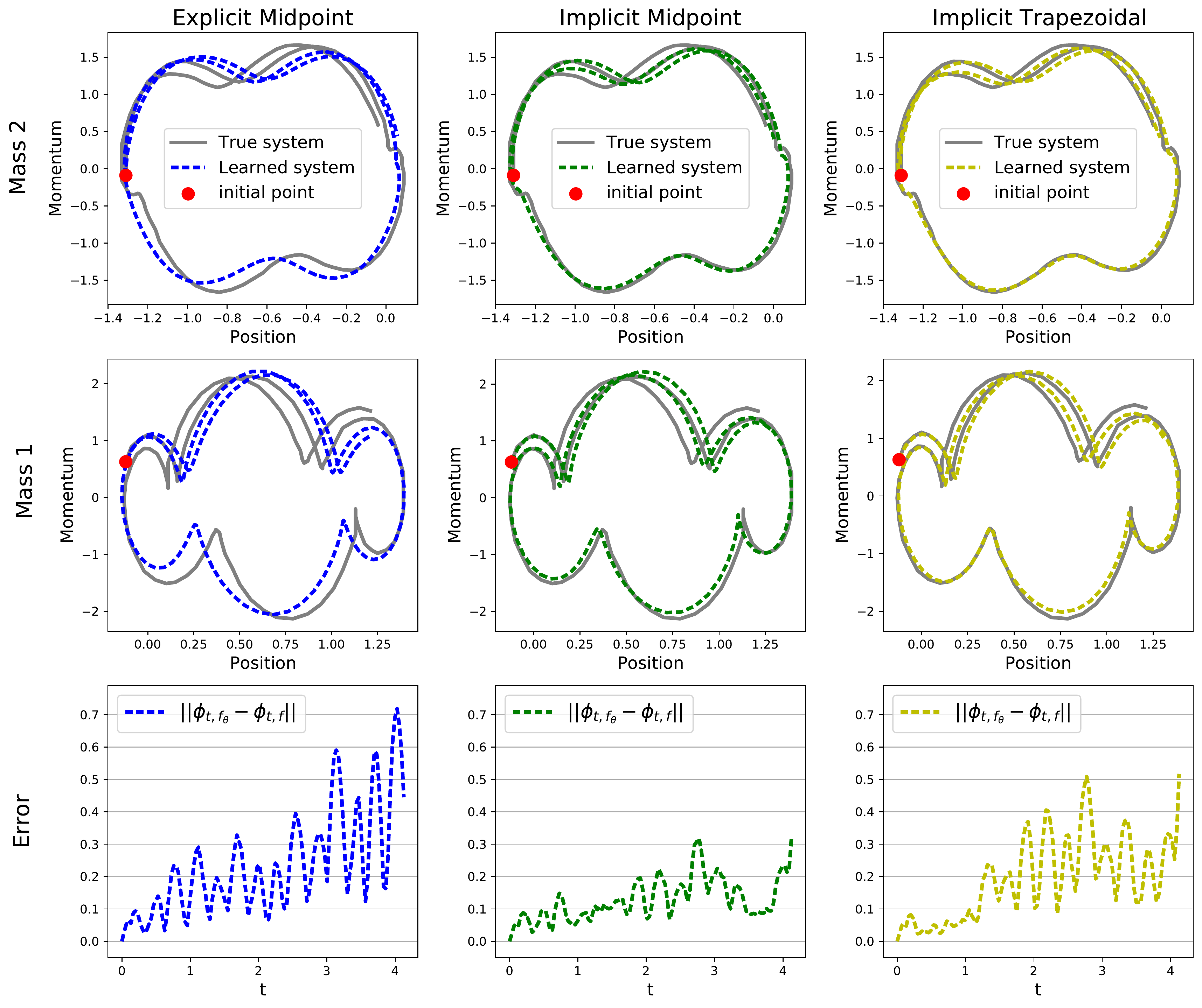}
    \caption{
        \textbf{Results of learning the real-world dynamics.}
        We find that the fine details of the trajectory
        around position $0.0$, momentum $0.1$
        are preserved more faithfully
        by the two implicit methods than by the one explicit.
        This effect is conserved across multiple experiments.
    }
    \label{fig:Real}
\end{figure}

We use the last point in the training data as the initial point to simulate the learned system, and depict the learned dynamics and test trajectory in \cref{fig:Real}. 
Despite the measurement errors and other non-ideal effects, we see that the proposed algorithm still performs robustly.
In addition, while all test schemes are of order $2$, the use of implicit schemes 
for identification preserves the phase portrait more accurately and the implicit midpoint method achieves the lowest prediction error. 
The use of implicit schemes also permits the incorporation of geometric properties such as symplecticity, symmetry and reversibility.
Although their necessity has not been mathematically proven, the results in \cref{fig:Real} show empirically better results with implicit training schemes.

\section*{Summary}
Machine learning via ODE-nets provides data-driven approaches to model and predict the dynamics of physical systems from data. Since the models are typically trained on discrete data, we have to perform a numerical integration to evaluate a loss for training. In this paper we extend previous work~\cite{zhu2022on}, in which we defined the inverse modified differential equation (IMDE). We prove that training an ODE-net templated on an unrolled implicit scheme returns an approximation of a particular IMDE. In addition, we show that the convergence with discrete step $h$ is of order $p$, where $p$ is the order of the numerical
integrator. Numerical experiments support the theoretical findings.

In addition, for learning with neural networks templated on implicit numerical integration, we propose and implement an adaptive algorithm that adjusts the iteration number of unrolled implicit integration during the training process to accelerate training. Instead of treating numerical integration of ODE-nets as a black box, our algorithm allows for finding the cheapest iteration number via monitoring the errors of the implicit solver and the learning loss. Numerical experiments show that the proposed algorithm leads to a $2-3\times$ speedup in training without any degradation in accuracy. Finally, we remark that our method naturally applies to the approaches based on ODE-net incorporating partially known physical terms (i.e., “gray box” identification). \cite{rico1994continuous,lovelett2020graybox}

Several challenges remain to be addressed in future work.
First, the Newton-Raphson iteration \cref{eq:irk_nr} requires solving a linear equation, which makes scaling to high-dimensional equation discovery expensive. One possible direction is to do Newton-Raphson steps with an iterative algorithm such as GMRES~\cite{saad1986gmres}. 

Second, our algorithm uses the interplay between the training and numerical integration to adapt the stopping criterion.
Such an idea can also be extended to efficient adaptive time-step methods, where the IMDE for adaptive steps still remains open.

Third, in classical initial value solvers, it is well known that implicit schemes have better stability~\cite{hairer1996solving}, and allow for geometric properties such as symplecticity, symmetry and 
reversibility~\cite{hairer2006geometric}.
We would like to further explore in future work how these well-known forward-integration properties of implicit methods 
produce benefits in implicitly-templated ODE-nets over the merely explicitly-templated ones.

Finally, while we provide a rigorous grounding for the proposed adaptive \cref{Training with adaptive iteration},
this suggests a family of further adaptive methods
for accelerating the neural identification of ODEs from data;
these should better exploit existing intuition and experience about the tradeoffs inherent in the methods available in the literature.
For instance, in the same way that the current learning loss sets a ceiling on the useful iteration number,
we might find that switching from unrolled ODE-nets to adjoint-differentiation NODEs does make sense, but possibly only later in the training process.
This program could be taken further towards a meta-learning approach,
where an agent is trained to make such hyperparameter decisions online, during the training of the target ODE network. \cite{Guo2021PersonalizedAG,doncevic2022recursively}

\appendix
\section{Calculation of IMDE}\label{app:Calculation of IMDE}
\subsection{Linear ODEs}
Consider a linear IVP
\begin{equation}
\label{eq:LinIVP}
\frac{d}{dt}\vec{y}(t) = \fRHS{}(\vec{y}(t)) = \mat{A}\vec{y}(t),\quad \vec{y}(0) =\vec{x},
\end{equation}
where $\vec{y}(t), \vec{b} \in \R^D$ and $\mat{A} \in \R^{D\times D}$ is invertible. It's solution at time $h$ can be given as
\begin{equation}
\label{eq:LinIVPInfSeriesIMDE}
\phi_{h,\fRHS{}}(\vec{x}) =  e^{\mat{A}h}\vec{x} = \sum_{k=0}^{\infty} \frac{h^k}{k!}\mat{A}^k \vec{x} = (\sum_{k=0}^{\infty} \frac{(-h)^k}{k!}\mat{A}^k)^{-1} \vec{x}.
\end{equation}
We now consider learning with implicit Euler scheme,
\begin{equation*}
\vec{v}_1 = \vec{x} + h \fRHS{}(\vec{v}_1),\quad \Phi_{h,\fRHS{}_h}(\vec{x}) = \vec{x} + h \fRHS{}_h(\vec{v}_1)
\end{equation*}
we have
\begin{equation*}
\Phi_{h,\fRHS{}_h}(\vec{x}) = (\mat{I}_D - h\fRHS{}_h)^{-1}(\vec{x})
\end{equation*}
where $\mat{I}_D$ is the identity map.
By $\Phi_{h,\fRHS{}_h} = \phi_{h,\fRHS{}}$, we deduce that
\begin{equation*}
\fRHS{}_h(\vec{x}) = \sum_{k=0}^{\infty} \frac{(-1)^{k}h^k}{(k+1)!}\mat{A}^{k+1}\vec{x} = \mat{A}_h \vec{x},
\end{equation*}
Additionally, if $\fRHS{}(\vec{y}) = \mat{A}\vec{y}+\vec{b}$, we have $\fRHS{}_h(\vec{x}) = \mat{A}_h\vec{x} + \mat{A}_h\mat{A}^{-1}\vec{b}$ by linear transformation $\vec{\hat{y}} = \vec{y}+\mat{A}^{-1}\vec{b}$.
\subsection{General nonlinear ODEs}
For a nonlinear ODE,
\begin{equation*}
\frac{d}{dt}\vec{y}(t) = \fRHS{}(\vec{y}(t)),\quad \vec{y}(0) =\vec{x},
\end{equation*}
we first expand the exact solution:
\begin{equation}\label{eq:exasolu_a}
\begin{aligned}
\phi_{h,\fRHS{}}(\vec{x})&=\vec{x}+h\fRHS{}(\vec{x})+\frac{h^2}{2}\fRHS{}'\fRHS{}(\vec{x})+\frac{h^3}{6}(\fRHS{}''(\fRHS{},\fRHS{})(\vec{x})+\fRHS{}'\fRHS{}'\fRHS{}(\vec{x}))\\
&+ \frac{h^4}{24}(\fRHS{}'''(\fRHS{},\fRHS{},\fRHS{})(\vec{x})+3\fRHS{}''(\fRHS{}'\fRHS{},\fRHS{})(\vec{x})+\fRHS{}'\fRHS{}''(\fRHS{},\fRHS{})(\vec{x})+\fRHS{}'\fRHS{}'\fRHS{}'\fRHS{}(\vec{x})) +\cdots .
\end{aligned}
\end{equation}
As an example, the numerical scheme is chosen to be implicit Euler scheme, using Newton-Raphson iteration with $L=1$,
\begin{equation*}
\vec{v}_1^0 \equiv \vec{x}, \quad
\vec{v}_1^1 = \vec{x} + h \fRHS{}_h(\vec{v}_1^0)+h\fRHS{}_h'(\vec{v}_1^0)(\vec{v}_1^1-\vec{v}_1^0),\quad\Phi_{h,\fRHS{}_h}^1(\vec{x}) \equiv \vec{x} + h \fRHS{}_h(\vec{v}_1^1).
\end{equation*}
We expand it as
\begin{equation*}
\begin{aligned}
\Phi_{h,\fRHS{}_h}^1(\vec{x})=& \vec{x}+h\fRHS{}_h(\vec{x}) + h^2\fRHS{}_h'\fRHS{}_h(\vec{x}) + \frac{h^3}{2}\fRHS{}_h''(\fRHS{}_h,\fRHS{}_h)(\vec{x})+h^3\fRHS{}_h'\fRHS{}_h'\fRHS{}_h(\vec{x}) \\
&+ \frac{h^4}{6}\fRHS{}_h'''(\fRHS{}_h,\fRHS{}_h,\fRHS{}_h)(\vec{x})+ h^4\fRHS{}_h''(\fRHS{}_h'\fRHS{}_h,\fRHS{}_h)(\vec{x}) + h^4\fRHS{}_h'\fRHS{}_h'\fRHS{}_h'\fRHS{}_h(\vec{x})+\cdots.\\
\end{aligned}
\end{equation*}
Substituting
\begin{equation*}
    \fRHS{}_h= \fRHS{}_0+h\fRHS{}_1+h^2\fRHS{}_2+h^3\fRHS{}_3\cdots
\end{equation*}
yields
\begin{equation*}
\begin{aligned}
\Phi_{h,\fRHS{}_h}^{2}(\vec{x})= &\vec{x} + h\fRHS{}_0+h^2(\fRHS{}_1(\vec{x})+ \fRHS{}_0'\fRHS{}_0(\vec{x}))\\
&+h^3(\fRHS{}_2(\vec{x})+\fRHS{}_1'\fRHS{}_0(\vec{x})+\fRHS{}_0'\fRHS{}_1(\vec{x})+\frac{1}{2}\fRHS{}_0''(\fRHS{}_0,\fRHS{}_0)(\vec{x}) + \fRHS{}_0'\fRHS{}_0'\fRHS{}_0(\vec{x})) \\
&+ h^4\big(\fRHS{}_3(\vec{x}) +\fRHS{}_1'\fRHS{}_1(\vec{x})+\fRHS{}_0'\fRHS{}_2(\vec{x}) +\fRHS{}_2'\fRHS{}_0(\vec{x})+ \frac{1}{2}\fRHS{}_1''(\fRHS{}_0,\fRHS{}_0)(\vec{x})+\fRHS{}_0''(\fRHS{}_1,\fRHS{}_0)(\vec{x})\\
&+\fRHS{}_1'\fRHS{}_0'\fRHS{}_0(\vec{x})+\fRHS{}_0'\fRHS{}_1'\fRHS{}_0(\vec{x})+\fRHS{}_0'\fRHS{}_0'\fRHS{}_1(\vec{x})\\
&+ \frac{1}{6}\fRHS{}_0'''(\fRHS{}_0,\fRHS{}_0,\fRHS{}_0)(\vec{x})+ \fRHS{}_0''(\fRHS{}_0'\fRHS{}_0,\fRHS{}_0)(\vec{x}) + \fRHS{}_0'\fRHS{}_0'\fRHS{}_0'\fRHS{}_0(\vec{x})\big)+ \cdots.
\end{aligned}
\end{equation*}
Comparing like powers of $h$ with expression \cref{eq:exasolu_a} yields recurrence relations for functions $\fRHS{}_k$, i.e.,
\begin{equation*}
\begin{aligned}
\fRHS{}_0(\vec{y}) &= \fRHS{}(\vec{y}),\\
\fRHS{}_1(\vec{y}) &= \frac{1}{2}\fRHS{}'\fRHS{}(\vec{y})-\fRHS{}_0'\fRHS{}_0(\vec{y}) = -\frac{1}{2}\fRHS{}'\fRHS{}(\vec{y}),\\
\fRHS{}_2(\vec{y}) &= \frac{1}{6}(\fRHS{}''(\fRHS{},\fRHS{})(\vec{y})+\fRHS{}'\fRHS{}'\fRHS{}(\vec{y})) - (\fRHS{}_1'\fRHS{}_0(\vec{y})+\fRHS{}_0'\fRHS{}_1(\vec{y})+\frac{1}{2}\fRHS{}_0''(\fRHS{}_0,\fRHS{}_0)(\vec{y}) + \fRHS{}_0'\fRHS{}_0'\fRHS{}_0(\vec{y}))\\
& = \frac{1}{6}\fRHS{}''(\fRHS{},\fRHS{})(\vec{y})+\frac{1}{6}\fRHS{}'\fRHS{}'\fRHS{}(\vec{y}),\\
\fRHS{}_3(\vec{y}) &=  \frac{1}{24}(\fRHS{}'''(\fRHS{},\fRHS{},\fRHS{})(\vec{y})+3\fRHS{}''(\fRHS{}'\fRHS{},\fRHS{})(\vec{y})+\fRHS{}'\fRHS{}''(\fRHS{},\fRHS{})(\vec{y})+\fRHS{}'\fRHS{}'\fRHS{}'\fRHS{}(\vec{y})) \\
&-\big(\fRHS{}_1'\fRHS{}_1(\vec{y})+\fRHS{}_0'\fRHS{}_2(\vec{y}) +\fRHS{}_2'\fRHS{}_0(\vec{y})+ \frac{1}{2}\fRHS{}_1''(\fRHS{}_0,\fRHS{}_0)(\vec{y})+\fRHS{}_0''(\fRHS{}_1,\fRHS{}_0)(\vec{y})\\
&+\fRHS{}_1'\fRHS{}_0'\fRHS{}_0(\vec{y})+\fRHS{}_0'\fRHS{}_1'\fRHS{}_0(\vec{y})+\fRHS{}_0'\fRHS{}_0'\fRHS{}_1(\vec{y})\\
&+ \frac{1}{6}\fRHS{}_0'''(\fRHS{}_0,\fRHS{}_0,\fRHS{}_0)(\vec{y})+ \fRHS{}_0''(\fRHS{}_0'\fRHS{}_0,\fRHS{}_0)(\vec{y}) + \fRHS{}_0'\fRHS{}_0'\fRHS{}_0'\fRHS{}_0(\vec{y})\big)\\
&=-\frac{1}{24}\fRHS{}'''(\fRHS{},\fRHS{},\fRHS{})(\vec{y})-\frac{1}{8}\fRHS{}''(\fRHS{}'\fRHS{},\fRHS{})(\vec{y})+\frac{11}{24}\fRHS{}'\fRHS{}''(\fRHS{},\fRHS{})(\vec{y})-\frac{1}{24}\fRHS{}'\fRHS{}'\fRHS{}'\fRHS{}(\vec{y}),\\
&\vdots
\end{aligned}
\end{equation*}
Note that Newton-Raphson iteration with $L = 1$ can exactly solve the implicit linear equation. In the linear case \cref{eq:LinIVP}, $\fRHS{}'=\mat{A}$ and all terms involving $\fRHS{}''$ and all higher-order derivatives are $0$, giving us
\begin{equation}
\label{eq:LinIVPInfSeriesIMDEExpandedTerms}
    \begin{aligned}
        \fRHS{}_0(\vec{y}) &= \mat{A} \cdot \vec{y} \quad\quad\quad &
        \fRHS{}_1(\vec{y}) &= -\frac{1}{2} \mat{A}^2 \cdot \vec{y} \\
        \fRHS{}_2(\vec{y}) &= \frac{1}{6} \mat{A}^3 \cdot \vec{y} \quad\quad\quad &
        \fRHS{}_3(\vec{y}) &= -\frac{1}{24} \mat{A}^4 \cdot \vec{y} \\
    \end{aligned}
\end{equation}
so we recover approximately \cref{eq:LinIVPInfSeriesIMDE} via
\begin{equation}
\label{eq:LinIVPInfSeriesIMDEExpanded}
    \begin{aligned}
        \fRHS{}_h (\vec{y})&= \fRHS{}_0(\vec{y}) + h \fRHS{}_1(\vec{y}) + h^2 \fRHS{}_2(\vec{y}) + h^3 \fRHS{}_3(\vec{y}) + \cdots
        \\
        & = \left(
            \frac{h^0}{1} \mat{A} + \frac{-h}{2} \mat{A}^2
            +\frac{h^2}{6} \mat{A}^3
            -\frac{1 h^3}{24} \mat{A}^4
            + \cdots
        \right) \cdot \vec{y}
        \\
        &\approx \sum_{k=0}^\infty \frac{(-1)^k h^k}{(k+1)!} \mat{A}^{k+1} \cdot \vec{y}.
    \end{aligned}
\end{equation}

We next present an example employing fixed-point iteration with $L=2$,
\begin{equation}
\label{eq:GenNonlinUnroll}
\vec{v}_1^0 \equiv \vec{x}, \quad
\vec{v}_1^1 = \vec{x} + h \fRHS{}(\vec{x}),\quad \vec{v}_1^2 = \vec{x} + h \fRHS{}(\vec{v}_1^1),\quad\Phi_{h,\fRHS{}_h}^2(\vec{x}) \equiv \vec{x} + h \fRHS{}_h(\vec{v}_1^2).
\end{equation}
We then use B-series to expand $\Phi_{h,\fRHS{}_h}^2(\vec{x})$:
\begin{equation*}
\begin{aligned}
\Phi_{h,\fRHS{}_h}^2(\vec{x})=& \vec{x}+h\fRHS{}_h(\vec{x}) + h^2\fRHS{}_h'\fRHS{}_h(\vec{x}) + \frac{h^3}{2}\fRHS{}_h''(\fRHS{}_h,\fRHS{}_h)(\vec{x})+h^3\fRHS{}_h'\fRHS{}_h'\fRHS{}_h(\vec{x}) \\
&+ \frac{h^4}{6}\fRHS{}_h'''(\fRHS{}_h,\fRHS{}_h,\fRHS{}_h)(\vec{x})+ h^4\fRHS{}_h''(\fRHS{}_h'\fRHS{}_h,\fRHS{}_h)(\vec{x}) + \frac{h^4}{2}\fRHS{}_h'\fRHS{}_h''(\fRHS{}_h,\fRHS{}_h)(\vec{x})+\cdots.\\
\end{aligned}
\end{equation*}
And similarly we have
\begin{equation*}
\begin{aligned}
\Phi_{h,\fRHS{}_h}^{2}(\vec{x})= &\vec{x} + h\fRHS{}_0+h^2(\fRHS{}_1(\vec{x})+ \fRHS{}_0'\fRHS{}_0(\vec{x}))\\
&+h^3(\fRHS{}_2(\vec{x})+\fRHS{}_1'\fRHS{}_0(\vec{x})+\fRHS{}_0'\fRHS{}_1(\vec{x})+\frac{1}{2}\fRHS{}_0''(\fRHS{}_0,\fRHS{}_0)(\vec{x}) + \fRHS{}_0'\fRHS{}_0'\fRHS{}_0(\vec{x})) \\
&+ h^4\big(\fRHS{}_3(\vec{x}) +\fRHS{}_1'\fRHS{}_1(\vec{x})+\fRHS{}_0'\fRHS{}_2(\vec{x}) +\fRHS{}_2'\fRHS{}_0(\vec{x})+ \frac{1}{2}\fRHS{}_1''(\fRHS{}_0,\fRHS{}_0)(\vec{x})+\fRHS{}_0''(\fRHS{}_1,\fRHS{}_0)(\vec{x})\\
&+\fRHS{}_1'\fRHS{}_0'\fRHS{}_0(\vec{x})+\fRHS{}_0'\fRHS{}_1'\fRHS{}_0(\vec{x})+\fRHS{}_0'\fRHS{}_0'\fRHS{}_1(\vec{x})\\
&+ \frac{1}{6}\fRHS{}_0'''(\fRHS{}_0,\fRHS{}_0,\fRHS{}_0)(\vec{x})+ \fRHS{}_0''(\fRHS{}_0'\fRHS{}_0,\fRHS{}_0)(\vec{x}) + \frac{1}{2}\fRHS{}_0'\fRHS{}_0''(\fRHS{}_0,\fRHS{}_0)(\vec{x})\big)+ \cdots.
\end{aligned}
\end{equation*}
Comparing like powers of $h$ with expression \cref{eq:exasolu_a}, we obtain that
\begin{equation*}
\begin{aligned}
\fRHS{}_0(\vec{y}) &= \fRHS{}(\vec{y}),\\
\fRHS{}_1(\vec{y}) &= \frac{1}{2}\fRHS{}'\fRHS{}(\vec{y})-\fRHS{}_0'\fRHS{}_0(\vec{y}) = -\frac{1}{2}\fRHS{}'\fRHS{}(\vec{y}),\\
\fRHS{}_2(\vec{y}) &= \frac{1}{6}(\fRHS{}''(\fRHS{},\fRHS{})(\vec{y})+\fRHS{}'\fRHS{}'\fRHS{}(\vec{y})) - (\fRHS{}_1'\fRHS{}_0(\vec{y})+\fRHS{}_0'\fRHS{}_1(\vec{y})+\frac{1}{2}\fRHS{}_0''(\fRHS{}_0,\fRHS{}_0)(\vec{y}) + \fRHS{}_0'\fRHS{}_0'\fRHS{}_0(\vec{y}))\\
& = \frac{1}{6}\fRHS{}''(\fRHS{},\fRHS{})(\vec{y})+\frac{1}{6}\fRHS{}'\fRHS{}'\fRHS{}(\vec{y}),\\
\fRHS{}_3(\vec{y}) &=  \frac{1}{24}(\fRHS{}'''(\fRHS{},\fRHS{},\fRHS{})(\vec{y})+3\fRHS{}''(\fRHS{}'\fRHS{},\fRHS{})(\vec{y})+\fRHS{}'\fRHS{}''(\fRHS{},\fRHS{})(\vec{y})+\fRHS{}'\fRHS{}'\fRHS{}'\fRHS{}(\vec{y})) \\
&-\big(\fRHS{}_1'\fRHS{}_1(\vec{y})+\fRHS{}_0'\fRHS{}_2(\vec{y}) +\fRHS{}_2'\fRHS{}_0(\vec{y})+ \frac{1}{2}\fRHS{}_1''(\fRHS{}_0,\fRHS{}_0)(\vec{y})+\fRHS{}_0''(\fRHS{}_1,\fRHS{}_0)(\vec{y})\\
&+\fRHS{}_1'\fRHS{}_0'\fRHS{}_0(\vec{y})+\fRHS{}_0'\fRHS{}_1'\fRHS{}_0(\vec{y})+\fRHS{}_0'\fRHS{}_0'\fRHS{}_1(\vec{y})\\
&+ \frac{1}{6}\fRHS{}_0'''(\fRHS{}_0,\fRHS{}_0,\fRHS{}_0)(\vec{y})+ \fRHS{}_0''(\fRHS{}_0'\fRHS{}_0,\fRHS{}_0)(\vec{y}) + \frac{1}{2}\fRHS{}_0'\fRHS{}_0''(\fRHS{}_0,\fRHS{}_0)(\vec{y})\big)\\
&=-\frac{1}{24}\fRHS{}'''(\fRHS{},\fRHS{},\fRHS{})(\vec{y})-\frac{1}{8}\fRHS{}''(\fRHS{}'\fRHS{},\fRHS{})(\vec{y})-\frac{1}{24}\fRHS{}'\fRHS{}''(\fRHS{},\fRHS{})(\vec{y})+\frac{23}{24}\fRHS{}'\fRHS{}'\fRHS{}'\fRHS{}(\vec{y}),\\
&\vdots
\end{aligned}
\end{equation*}

\section{Proofs} 
\subsection{Proof of Theorem \ref{the:implicit imde} (\nameref{the:implicit imde})}
\label{app:the:implicit imde}
The proof of \cref{the:implicit imde} is obtained as the proof of Theorem 3.1 
in \cite{zhu2022on} under the following \cref{asm:int}. Here we sketch the main idea in the notation used there, and then show that \cref{asm:int} holds.

\begin{namedAssumption}[Assumptions for numerical schemes]\label{asm:int}
For analytic $\gRHS{}$, $\hat{\gRHS}$ satisfying $\norm{\gRHS{}}_{\mathcal{B}(\mathcal{K},r)}\leq m$, $\norm{\hat{\gRHS}}_{\mathcal{B}(\mathcal{K},r)} \leq m$, there exist constants $b_1,b_2,b_3$ that depend only on the
scheme $\Phi_{h}$ 
and composition number $S$ such that the unrolled approximation $\Phi^L_{h}$ satisfies
\begin{itemize}
    \item for $|h|\leq h_0 = b_{1} r/m$ , $(\Phi^L_{h,\hat{\gRHS}})^S$, $(\Phi^L_{h,\gRHS{}})^S$ are analytic on $\mathcal{K}$.
    \item for $|h|\leq h_0$,
    \begin{equation*}
        \norm{\brac{\Phi^L_{h,\hat{\gRHS}}}^S-\brac{\Phi^L_{h,\gRHS{}}}^S}_{\mathcal{K}}\leq b_{2}S|h|\norm{\hat{\gRHS}-\gRHS{}}_{\mathcal{B}(\mathcal{K},r)}.
    \end{equation*}
    \item for $|h|< h_1 < h_0$,
    \begin{equation*}
    \begin{aligned}
    \norm{\hat{\gRHS}-\gRHS{}}_{\mathcal{K}} \leq & \frac{1}{S|h|}\norm{\brac{\Phi^L_{h,\hat{\gRHS}}}^S-\brac{\Phi^L_{h,\gRHS{}}}^S}_{\mathcal{K}} + \frac{b_2|h|}{h_1-|h|} \norm{\hat{\gRHS}-\gRHS{}}_{\mathcal{B}(\mathcal{K},b_3Sh_1 m)}.
    \end{aligned}
    \end{equation*}
\end{itemize}
\end{namedAssumption}
\begin{namedLemma}[Choice of truncation and estimation of error for IMDE]\label{lem:difference}
Let $\fRHS{}(\vec{y})$ be analytic in $\mathcal{B}(\mathcal{K}, r)$ and satisfies $\norm{\fRHS{}}_{\mathcal{B}(\mathcal{K}, r)} \leq m$. Suppose the numerical scheme $\Phi_{h}$ and its approximation $\Phi^L_{h}$ satisfy \cref{asm:int}. Take $\eta = \max\{6, \frac{b_2+1}{29}+1\}$, $\zeta = 10(\eta-1)$, $q = -\ln(2 b_2)/ \ln 0.912$ and let $K$ be the largest integer satisfying
\begin{equation*}
\frac{\zeta(K-p+2)^q|h|m}{\eta r} \leq e^{-q}.
\end{equation*}
If $|h|$ is small enough such that $K\geq p$, then the truncated IMDE satisfies
\begin{equation*}
\begin{aligned}
&\norm{(\Phi^L_{h,\fRHS{}_h^{K}})^S-\phi_{Sh,\fRHS{}}}_{\mathcal{K}} \leq b_2\eta m e^{2q-qp}|Sh|e^{-\gamma /|Sh|^{1/q}},\\
&\norm{\sum\nolimits_{k=p}\nolimits^Kh^k\fRHS{}_k}_{\mathcal{K}}\leq b_2 \eta m \brac{\frac{\zeta m}{b_1r }}^p (1+ 1.38^q d_p) |h|^p,\\
& \norm{\fRHS{}_{h}^K}_{\mathcal{K}} \leq (\eta-1)m,
\end{aligned}
\end{equation*}
where $\gamma = \frac{q}{e}\brac{\frac{b_1r}{\zeta m}}^{1/q}$, $d_p = p^{qp} e^{-q(p-1)}$.
\end{namedLemma}
\begin{proof}
According the Lemma B.1 in \cite{zhu2022on}, the IMDE of the $(\Phi^L_{h})^S$ coincides with the IMDE of $\Phi^L_{h}$. In addition, via regarding compositions $(\Phi^L_{h})^S$ as a one-step integrator, the estimates are obtained as in the proof of Lemma B.4 in \cite{zhu2022on}. Special constants including $0.912$ and $1.38$
are also explained in \cite{zhu2022on}.
\end{proof}
\begin{proof}[Proof of \cref{the:implicit imde}]
According to \cref{lem:difference}, we have that
\begin{equation}\label{Phi fnet-fh}
\delta : = \frac{1}{\Delta t}\norm{\brac{\Phi_{h, \fRHS{}_{\theta}}}^S-\brac{\Phi_{h,\fRHS{}_h^K}}^S}_{\mathcal{B}(\vec{x}, r_1)} \leq \mathcal{L} + c m e^{-\gamma /\Delta t^{1/q}},\quad \norm{\fRHS{}_{h}^K}_{\mathcal{B}(\vec{x}, r_1)}<(\eta-1)m,
\end{equation}
where $c = b_2\eta e^q$. Let
\begin{equation*}
h_1 = (eb_2+1)h, \ M = (\eta-1)m.
\end{equation*}
By the third term of \cref{asm:int}, we deduce that for $0\leq j \leq r_1/b_3Sh_1 M$,
\begin{equation*}
\begin{aligned}
\norm{\fRHS{}_{\theta} - \fRHS{}_h^K}_{\mathcal{B}(\vec{x}, jb_3Sh_1M)} \leq \delta + e^{-1} \norm{\fRHS{}_{\theta} - \fRHS{}_h^K}_{\mathcal{B}(\vec{x}, (j+1)b_3Sh_1 M)}.
\end{aligned}
\end{equation*}
As in the proof of Theorem 3.1 in \cite{zhu2022on}. we obtain that
\begin{equation*}\label{fnet-fh}
\begin{aligned}
\norm{\fRHS{}_{\theta}(\vec{x}) - \fRHS{}_h^K(\vec{x})} \leq e^{-\hat{\gamma}/\Delta t} \norm{\fRHS{}_{\theta} - \fRHS{}_h^K}_{\mathcal{B}(\vec{x}, r_1)} + \frac{\delta}{1-\lambda},
\end{aligned}
\end{equation*}
where $\hat{\gamma} = \frac{r_1}{(eb_2+1)b_3M}$. And thus we conclude that
\begin{equation*}
\norm{\fRHS{}_{\theta}(\vec{x}) - \fRHS{}_h^K(\vec{x})} \leq c_1 m e^{-\gamma /\Delta t^{1/q}} + C \mathcal{L},
\end{equation*}
where $C =e/(e-1)$ and $c_1$ is a constant satisfying $c_1\geq C \cdot c+ \eta e^{\gamma/\Delta t^{1/q} - \hat{\gamma}/\Delta t}$.
\end{proof}

Next, to complete the proof of \cref{the:implicit imde}, it suffices to show that unrolled implicit Runge-Kutta scheme $\Phi_h$ \cref{eq:rk} using fixed-point iteration \cref{eq:irk} or Newton-Raphson iteration \cref{eq:irk_nr} both satisfy \cref{asm:int}.
\begin{namedLemma}[Fixed-point iteration obeys \cref{asm:int}]\label{lem:rkfp}
Consider a consistent implicit Runge-Kutta scheme $\Phi_h$ \cref{eq:rk} and its approximation via fixed-point iteration $\Phi_h^{L}$ \cref{eq:irk}, denote
\begin{equation*}
\mu = \sum_{i=1}^I|b_i|,  \quad \kappa =\max_{1\leq i\leq I}\sum_{j=1}^I|a_{ij}|.
\end{equation*}
Let $\gRHS{}, \hat{\gRHS}$ be analytic in $\mathcal{B}(\mathcal{K},r)$
and satisfy $\norm{\gRHS{}}_{\mathcal{B}(\mathcal{K},r)}\leq m$, $\norm{\hat{\gRHS}}_{\mathcal{B}(\mathcal{K},r)} \leq m$.
Then, for $|h|\leq h_0 = r/(2(S\mu +\kappa)m)$ and $\vec{x} \in \mathcal{K}$, the compositions $(\Phi_{h,\gRHS{}}^{L})^S(\vec{x})$, $(\Phi_{h,\hat{\gRHS}}^{L})^S(\vec{x})$ are analytic and
\begin{equation*}
\norm{(\Phi_{h,\hat{\gRHS}}^{L})^S-(\Phi_{h,\gRHS{}}^{L})^S}_{\mathcal{K}}\leq (e-1)(S\mu+\kappa)|h| \norm{\hat{\gRHS}-\gRHS{}}_{\mathcal{B}(\mathcal{K},r)}.
\end{equation*}
In addition, for $|h|< h_1 \leq h_0$,
\begin{equation*}
\norm{\hat{\gRHS}-\gRHS{}}_{\mathcal{K}} \leq \frac{\norm{(\Phi_{h,\hat{\gRHS}}^{L})^S-(\Phi_{h,\gRHS{}}^{L})^S}_{\mathcal{K}}}{S|h|} +\frac{(e-1)(\mu + \kappa/S) |h|\norm{\hat{\gRHS}-\gRHS{}}_{\mathcal{B}(\mathcal{K},(S\mu + \kappa)h_1 m)}}{h_1-|h|}.
\end{equation*}
\end{namedLemma}
\begin{proof}
For $\vec{y} \in \mathcal{B}(\mathcal{K}, r/2)$ and $\norm{\Delta \vec{y}} \leq 1$, the function $\alpha(z) = \gRHS{}(\vec{y}+ z \Delta \vec{y})$ is analytic  for $|z| \leq r/2$ and bounded by $m$. By Cauchy's estimate, we obtain
\begin{equation*}
\norm{\gRHS{}'(\vec{y})\Delta \vec{y}} = \norm{\alpha ' (0)} \leq 2m/r,
\end{equation*}
and $\norm{\gRHS{}'(\vec{y})}\leq 2m/r$ for $y \in \mathcal{B}(\mathcal{K}, r/2)$ in the operator norm. Similarly, $\norm{\hat{\gRHS}'(y)}\leq 2m/r$ for $y \in \mathcal{B}(\mathcal{K}, r/2)$.

As in \cref{eq:GenNonlinUnroll}, when
using fixed-point iteration to unroll Runge-Kutta method \cref{eq:irk}, the solutions are recursively obtained by
\begin{equation*}
\left\{\begin{aligned}
&\vec{v}_i^{0,s} = \vec{x}^{L,s}, \quad \vec{x}^{L,0} = \vec{x}\\
&\vec{v}_i^{l,s} = \vec{x}^{L,s} + h\sum_{j=1}^I a_{ij} \gRHS{}(\vec{v}_j^{l-1,s}),\\
&\vec{x}^{L,s+1} = \vec{x}^{L,s} + h\sum_{i=1}^I b_{i} \gRHS{}(\vec{v}_i^{L,s}),\\
\end{aligned}\right.
\quad
\left\{\begin{aligned}
&\vec{\hat{v}}_i^{0,s} = \vec{\hat{x}}^{L,s}, \quad \vec{\hat{x}}^{L,0} = \vec{x}\\
&\vec{\hat{v}}_i^{l,s} = \vec{\hat{x}}^{L,s} + h\sum_{j=1}^I a_{ij}  \hat{\gRHS}(\vec{\hat{v}}_j^{l-1,s}),\\
&\vec{\hat{x}}^{L,s+1} = \vec{\hat{x}}^{L,s} + h\sum_{i=1}^I b_{i} \hat{\gRHS} (\vec{\hat{v}}_i^{L,s}),
\end{aligned}\right.
\end{equation*}
where $l=1, \cdots, L$, $s=0, \cdots, S-1$. For $|h|\leq h_0 = r/(2(S\mu +\kappa)m)$ and $\vec{x} \in \mathcal{K}$, we can readily check that
\begin{equation*}
\begin{aligned}
&\vec{x}^{L,s},\vec{\hat{x}}^{L,s} \in \mathcal{B}(\mathcal{K},s\mu m |h|) \ \text{for } s=0,\cdots, S,\\
&\vec{v}^{l,s}_i,\vec{\hat{v}}^{l,s}_i \in \mathcal{B}(\mathcal{K},(s\mu  + \kappa)  m |h|) \ \text{for } s=0,\cdots, S-1,\ l=1,\cdots, L,\ i = 1,\cdots, I.\\
\end{aligned}
\end{equation*}
Denote $V^{l,s} = \max_{1\leq i\leq I}\|\vec{v}_i^{l,s} - \vec{\hat{v}}_i^{l,s}\|$, $X^{s} = \|\vec{x}^{L,s} - \vec{\hat{x}}^{L,s}\|$, we have
\begin{equation*}
\begin{aligned}
\|\vec{v}_i^{l,s} - \vec{\hat{v}}_i^{l,s}\|
\leq & |h|\sum_{j=1}^s|a_{ij}|(\|\gRHS{}(\vec{v}_j^{l-1,s})-\gRHS{}(\vec{\hat{v}}_j^{l-1,s})\| + \|\gRHS{}(\vec{\hat{v}}_j^{l-1,s})-\hat{\gRHS}(\vec{\hat{v}}_j^{l-1,s})\|)+ X^s\\
\leq & |h|\kappa \frac{2m}{r}V^{l-1,s}+|h|\kappa\|\hat{\gRHS}-g\|_{\mathcal{B}(\mathcal{K},(S\mu+\kappa)  m |h|)} + X^s.
\end{aligned}
\end{equation*}
Thus we obtain
\begin{equation*}
V^{l,s} \leq |h|\kappa \frac{2m}{r}V^{l-1,s}+|h|\kappa\norm{\hat{\gRHS}-\gRHS{}}_{\mathcal{B}(\mathcal{K},(S\mu+\kappa)  m |h|)}+ X^s.
\end{equation*}
As a result, we have
\begin{equation}\label{eq:uls}
\begin{aligned}
V^{L,s}\leq& (|h|\kappa \frac{2m}{r})^LV^{0,s} + \frac{1-(|h|\kappa \frac{2m}{r})^L}{1-|h|\kappa \frac{2m}{r}}X^s+ \frac{1-(|h|\kappa \frac{2m}{r})^L}{1-|h|\kappa \frac{2m}{r}}|h|\kappa \norm{\hat{\gRHS}-\gRHS{}}_{\mathcal{B}(\mathcal{K},(S\mu+\kappa)  m |h|)}\\
\leq & \frac{1}{1-|h|\kappa \frac{2m}{r}}X^s+ \frac{1}{1-|h|\kappa \frac{2m}{r}}|h|\kappa \norm{\hat{\gRHS}-\gRHS{}}_{\mathcal{B}(\mathcal{K},(S\mu+\kappa)  m |h|)}.
\end{aligned}
\end{equation}
In addition,
\begin{equation}
\label{eq:est2}
\begin{aligned}
\|\vec{x}^{L,s+1} - \vec{\hat{x}}^{L,s+1}\|
\leq &X^s+ |h|\sum_{i=1}^s|b_i| (\|\gRHS{}(\vec{v}_j^{L,s})-\gRHS{}(\vec{\hat{v}}_j^{L,s})\| + \|\gRHS{}(\vec{\hat{v}}_j^{L,s})-\hat{\gRHS}(\vec{\hat{v}}_j^{L,s})\|)\\
\leq & X^s+ |h|\mu \frac{2m}{r}V^{L,s}+|h|\mu\norm{\hat{\gRHS}-\gRHS{}}_{\mathcal{B}(\mathcal{K},(S\mu+\kappa)  m |h|)}.
\end{aligned}
\end{equation}
These estimates, together with \cref{eq:uls}, indicate that
\begin{equation}\label{eq:vs}
X^{s+1} \leq (1+\frac{|h|\mu\frac{2m}{r}}{1-|h|\kappa \frac{2m}{r}})X^s+ ( \frac{|h|\mu \frac{2m}{r}}{1-|h|\kappa \frac{2m}{r}}\kappa + \mu)|h| \norm{\hat{\gRHS}-\gRHS{}}_{\mathcal{B}(\mathcal{K},(S\mu+\kappa)  m |h|)},
\end{equation}
Therefore, we deduce that
\begin{equation*}
\begin{aligned}
X^S\leq & \frac{(1+\frac{|h|\mu\frac{2m}{r}}{1-|h|\kappa \frac{2m}{r}})^S-1}{\frac{|h|\mu\frac{2m}{r}}{1-|h|\kappa \frac{2m}{r}}} ( \frac{|h|\mu \frac{2m}{r}}{1-|h|\kappa \frac{2m}{r}}\kappa + \mu)|h| \norm{\hat{\gRHS}-\gRHS{}}_{\mathcal{B}(\mathcal{K},(S\mu+\kappa)  m |h|)}\\
\leq & (e-1)(S\mu+\kappa)|h| \norm{\hat{\gRHS}-\gRHS{}}_{\mathcal{B}(\mathcal{K},(S\mu+\kappa)  m |h|)}.
\end{aligned}
\end{equation*}
where we have used  the fact $|h|\mu \frac{2m}{r}/(1-|h|\kappa \frac{2m}{r}) \leq 1/S$. 

Finally, using Cauchy's estimate, we deduce that $h_1 \leq h_0$
\begin{equation*}
\begin{aligned}
\norm{\frac{d^i}{dh^i}\left((\Phi_{h,\hat{\gRHS}}^{L})^S(\vec{x})-(\Phi_{h,\gRHS{}}^{L})^S(\vec{x})\right)\big|_{h=0}}
\leq \frac{ i!\cdot (e-1)(S\mu + \kappa) \norm{\hat{\gRHS}-\gRHS{}}_{\mathcal{B}(\mathcal{K},(S\mu + \kappa) h_1 m)}}{h_1^{i-1}}.
\end{aligned}
\end{equation*}
By the analyticity and triangle inequality, we obtain that for $|h|< h_1$,
\begin{equation*}
\begin{aligned}
&\norm{(\Phi_{h,\hat{\gRHS}}^{L})^S(\vec{x})-(\Phi_{h,\gRHS{}}^{L})^S(\vec{x})}\\
\geq& S|h|\norm{\hat{\gRHS}(\vec{x})-\gRHS{}(\vec{x})} - \sum_{i=2}^{\infty}\norm{\frac{h^i}{i!}\frac{d^j}{dh^j}\brac{(\Phi_{h,\hat{\gRHS}}^{L})^S(\vec{x})-(\Phi_{h,\gRHS{}}^{L})^S(\vec{x})}\big|_{h=0}}\\
\geq& S|h|\norm{\hat{\gRHS}(\vec{x})-\gRHS{}(\vec{x})} - (e-1)(S\mu + \kappa)|h| \norm{\hat{\gRHS}-\gRHS{}}_{\mathcal{B}(\mathcal{K},(S\mu + \kappa) h_1 m)}\sum_{i=2}^{\infty} \left(\frac{|h|}{h_1}\right)^{i-1}. \\
\end{aligned}
\end{equation*}
Therefore, we have
\begin{equation}
\label{eq:RHSerrFP}
\norm{\hat{\gRHS}-\gRHS{}}_{\mathcal{K}} \leq \frac{\norm{(\Phi_{h,\hat{\gRHS}}^{L})^S-(\Phi_{h,\gRHS{}}^{L})^S}_{\mathcal{K}}}{S|h|} +\frac{(e-1)(\mu + \kappa/S) |h|\norm{\hat{\gRHS}-\gRHS{}}_{\mathcal{B}(\mathcal{K},(S\mu + \kappa)h_1 m)}}{h_1-|h|},
\end{equation}
which concludes the proof.
\end{proof}

\begin{namedLemma}[Newton-Raphson iteration obeys \cref{asm:int}]\label{lem:rknr}
Consider a consistent implicit Runge-Kutta scheme $\Phi_h$ \cref{eq:rk}, and its approximation using Newton-Raphson iteration $\Phi_h^{L}$ \cref{eq:irk_nr}, denote
\begin{equation*}
\mu = \sum_{i=1}^I|b_i|,  \quad \kappa =\max_{1\leq i\leq I}\sum_{j=1}^I|a_{ij}|.
\end{equation*}
Let $\gRHS{}, \hat{\gRHS}$ be analytic in $\mathcal{B}(\mathcal{K},r)$ and satisfy $\norm{\gRHS{}}_{\mathcal{B}(\mathcal{K},r)}\leq m$, $\norm{\hat{\gRHS}}_{\mathcal{B}(\mathcal{K},r)} \leq m$. Then, for $|h|\leq h_0 = r/(2(S\mu +3.5\kappa)m)$ 
and $\vec{x} \in \mathcal{K}$, the compositions $(\Phi_{h,\gRHS{}}^{L})^S(\vec{x})$, $(\Phi_{h,\hat{\gRHS}}^{L})^S(\vec{x})$ are analytic and
\begin{equation*}
\norm{(\Phi_{h,\hat{\gRHS}}^{L})^S-(\Phi_{h,\gRHS{}}^{L})^S}_{\mathcal{K}}\leq (e-1) \mu S |h| \norm{\hat{\gRHS}-\gRHS{}}_{\mathcal{B}(\mathcal{K},r)}.
\end{equation*}
In addition, for $|h|< h_1 \leq h_0$,
\begin{equation}
\label{eq:RHSerrNR}
\norm{\hat{\gRHS}-\gRHS{}}_{\mathcal{K}} \leq \frac{\norm{(\Phi_{h,\hat{\gRHS}}^{L})^S-(\Phi_{h,\gRHS{}}^{L})^S}_{\mathcal{K}}}{S|h|} +\frac{(e-1)(\mu + \kappa/S) |h|\norm{\hat{\gRHS}-\gRHS{}}_{\mathcal{B}(\mathcal{K},(S\mu + 3\kappa)h_1 m)}}{h_1-|h|}.
\end{equation}
\end{namedLemma}
We note that \cref{eq:RHSerrFP} and \cref{eq:RHSerrNR}
differ only in the constants $3.5$ and $3$ used
(which are both $1$ in the FP case).
\begin{proof}
For $\vec{y} \in \mathcal{B}(\mathcal{K}, r/2)$ and $\norm{\Delta \vec{y}} \leq 1$, the function $\alpha(z) = \gRHS{}(\vec{y}+ z \Delta \vec{y})$ is analytic  for $|z| \leq r/2$ and bounded by $m$. By Cauchy's estimate, we obtain
\begin{equation*}
\norm{\gRHS{}'(\vec{y})\Delta \vec{y}} = \norm{\alpha ' (0)} \leq 2m/r,\quad \norm{\gRHS{}''(\vec{y})(\Delta \vec{y}, \Delta \vec{y})} = \norm{\alpha''(0)} \leq 4m/r^2
\end{equation*}
and thus $\norm{\gRHS{}'(\vec{y})}\leq 2m/r$, $\norm{\gRHS{}''(\vec{y})}\leq 4m/r^2$ for $\vec{y} \in \mathcal{B}(\mathcal{K}, r/2)$ in the operator norm. Similar estimates hold for $\hat{\gRHS}$.

When using Newton-Raphson iteration to unroll Runge-Kutta method \cref{eq:irk}, the solutions are recursively obtained by
\begin{equation*}
\left\{\begin{aligned}
&\vec{v}_i^{0,s} = \vec{x}^{L,s}, \quad \vec{x}^{L,0} = \vec{x}\\
&\vec{v}_i^{l,s} = \vec{x}^{L,s} + h\sum_{j=1}^I a_{ij} \big(\gRHS{}(\vec{v}_{j}^{l-1,s})+ g'(\vec{v}_{j}^{l-1,s})(\vec{v}_{j}^{l,s} - \vec{v}_{j}^{l-1,s})\big),\\
&\vec{x}^{L,s+1} = \vec{x}^{L,s} + h\sum_{i=1}^I b_{i} \gRHS{}(\vec{v}_i^{L,s}),\\
\end{aligned}\right.
\end{equation*}
\begin{equation*}
\left\{\begin{aligned}
&\vec{\hat{v}}_i^{0,s} = \vec{\hat{x}}^{L,s}, \quad \vec{\hat{x}}^{L,0} = \vec{x}\\
&\vec{\hat{v}}_i^{l,s} = \vec{\hat{x}}^{L,s} + h\sum_{j=1}^I a_{ij}  \big(\hat{\gRHS}(\vec{\hat{v}}_{j}^{l-1,s})+ \hat{\gRHS}'(\vec{\hat{v}}_{j}^{l-1,s})(\vec{\hat{v}}_{j}^{l,s} - \vec{\hat{v}}_{j}^{l-1,s})\big),\\
&\vec{\hat{x}}^{L,s+1} = \vec{\hat{x}}^{L,s} + h\sum_{i=1}^I b_{i} \hat{\gRHS} (\vec{\hat{v}}_i^{L,s}),\\
\end{aligned}\right.
\end{equation*}
where $l=1, \cdots, L$, $s=0, \cdots, S-1$. Denote $U^{l,s} = \max_{1\leq i\leq I}\norm{\vec{v}_i^{l.s} - \vec{v}_i^{l-1,s}}$, we have
\begin{equation*}
\begin{aligned}
&\norm{\vec{v}_i^{l,s} - \vec{v}_i^{l-1,s}}\\
\leq &  h\sum_{j=1}^I |a_{ij}| \norm{\gRHS{}'(\vec{v}_{j}^{l-1,s})(\vec{v}_i^{l,s} - \vec{v}_{j}^{l-1,s}) + \gRHS{}(\vec{v}_{j}^{l-1,s}) - \gRHS{}(\vec{v}_{j}^{l-2,s})- g'(\vec{v}_{j}^{l-2,s})(\vec{v}_i^{l-1,s} - \vec{v}_{j}^{l-2,s})}\\
\leq & m_1\kappa h U^{l,s} + m_2\kappa h (U^{l-1,s})^2\\
\end{aligned}
\end{equation*}
where $m_1 = \max_{s,l,i} \norm{\gRHS{}'(\vec{v}_{j}^{l-1,s})}$, $m_2 = \max_{s,l,i} \sup_{\theta\in [0,1]}\norm{\gRHS{}''(\theta \vec{v}_{j}^{l-1,s} + (1-\theta)\vec{v}_j^{l-2,s})}$.
Let $m_0 = \max_{s,i}{|\gRHS{}(\vec{x}^{L,s})|}$, we have
\begin{equation*}
U^{l,s} \leq \frac{m_2\kappa h }{1-m_1\kappa h}(U^{l-1,s})^2 \leq (\frac{m_2\kappa h }{1-m_1\kappa h})^{2^{l-1}-1}(U^{1,s})^{2^{l-1}} \leq (\frac{m_2\kappa h }{1-m_1\kappa h})^{2^{l-1}-1}(\frac{m_0\kappa h }{1-m_1\kappa h})^{2^{l-1}}.
\end{equation*}
Therefore, we can inductively check that for $s=0,\cdots, S-1$, $l=1,\cdots, L$, $i = 1,\cdots, I$,
\begin{equation*}
\begin{aligned}
&\vec{x}^{L,s},\vec{\hat{x}}^{L,s} \in \mathcal{B}(\mathcal{K},(s-1)|h|\mu m),\
&\vec{v}^{l,s}_i,\vec{\hat{v}}^{l,s}_i \in \mathcal{B}(\mathcal{K},(s-1)|h|\mu m + |h|\kappa m/(1-m_1\kappa h)),\\
& m_0\leq m,\quad m_1\leq 2m/r.
\end{aligned}
\end{equation*}
Denote $V^{l,s} = \max_{1\leq i\leq I}\|\vec{v}_i^{l,s} - \vec{\hat{v}}_i^{l,s}\|$, $X^{s} = \|\vec{x}^{L,s} - \vec{\hat{x}}^{s}\|$, we have
\begin{equation*}
\begin{aligned}
&\norm{\gRHS{}'(\vec{v}_{j}^{l-1,s})(\vec{v}_{j}^{l,s} - \vec{v}_{j}^{l-1,s}) - \hat{\gRHS}'(\vec{\hat{v}}_{j}^{l-1,s})(\vec{\hat{v}}_{j}^{l,s} - \vec{\hat{v}}_{j}^{l-1,s})}\\
\leq & \norm{\gRHS{}'(\vec{v}_{j}^{l-1,s})(\vec{v}_{j}^{l,s} - \vec{v}_{j}^{l-1,s}) - \hat{\gRHS}'(\vec{v}_{j}^{l-1,s})(\vec{v}_{j}^{l,s} - \vec{v}_{j}^{l-1,s})}\\
&+ \norm{\hat{\gRHS}'(\vec{v}_{j}^{l-1,s})(\vec{v}_{j}^{l,s} - \vec{v}_{j}^{l-1,s})- \hat{\gRHS}'(\vec{\hat{v}}_{j}^{l-1,s})(\vec{v}_{j}^{l,s} - \vec{v}_{j}^{l-1,s})}\\
&+ \norm{\hat{\gRHS}'(\vec{\hat{v}}_{j}^{l-1,s})(\vec{v}_{j}^{l,s} - \vec{v}_{j}^{l-1,s})- \hat{\gRHS}'(\vec{\hat{v}}_{j}^{l-1,s})(\vec{\hat{v}}_{j}^{l,s} - \vec{\hat{v}}_{j}^{l-1,s})}\\
\leq & \norm{\gRHS{}-\hat{\gRHS}}_{\mathcal{B}(\mathcal{K},(s-1)|h|\mu m + 2|h|\kappa m/(1-m_1\kappa |h|))} + \frac{4|h|\kappa m^2}{r^2(1-m_1\kappa |h|)}V^{l-1,s} +\frac{2m}{r}(V^{l-1,s}+V^{l,s})\\
\leq &\norm{\gRHS{}-\hat{\gRHS}}_{\mathcal{B}(\mathcal{K},(s-1)|h|\mu m + 3|h|\kappa m)} + \frac{m}{r}V^{l-1,s} +\frac{2m}{r}(V^{l-1,s}+V^{l,s}),
\end{aligned}
\end{equation*}
where the last inequality holds by the fact that $3\kappa m |h| \leq r/2$.

Subsequently, we deduce that
\begin{small}
\begin{equation*}
\begin{aligned}
&\|\vec{v}_i^{l,s} - \vec{\hat{v}}_i^{l,s}\|\\
\leq & X^s + |h|\sum_{j=1}^s|a_{ij}|(\|\gRHS{}(\vec{v}_j^{l-1,s})-\hat{\gRHS}(\vec{\hat{v}}_j^{l-1,s})\|+ \norm{\gRHS{}'(\vec{v}_{j}^{l-1,s})(\vec{v}_{j}^{l,s} - \vec{v}_{j}^{l-1,s}) - \hat{\gRHS}'(\vec{\hat{v}}_{j}^{l-1,s})(\vec{\hat{v}}_{j}^{l,s} - \vec{\hat{v}}_{j}^{l-1,s})})\\
\leq & |h|\kappa \frac{5m}{r}V^{l-1,s}+ |h|\kappa \frac{2m}{r}V^{l,s} +2|h|\kappa\|\hat{\gRHS}-g\|_{\mathcal{B}(\mathcal{K},(s-1)|h|\mu m + 3|h|\kappa m)} + X^s.
\end{aligned}
\end{equation*}
\end{small}
Thus we obtain
\begin{equation*}
V^{l,s} \leq \frac{|h|\kappa \frac{5m}{r}}{1-|h|\kappa \frac{2m}{r}}V^{l-1,s}+\frac{|h|\kappa\norm{\hat{\gRHS}-\gRHS{}}_{\mathcal{B}(\mathcal{K},(s-1)|h|\mu m + 3|h|\kappa m)}+ X^s}{1-|h|\kappa \frac{2m}{r}}.
\end{equation*}
As a result, we have
\begin{equation}\label{eq:uls_n}
\begin{aligned}
V^{L,s} \leq& \frac{1}{1-|h|\kappa \frac{7m}{r}}X^s+ \frac{1}{1-|h|\kappa \frac{7m}{r}}|h|\kappa \norm{\hat{\gRHS}-\gRHS{}}_{\mathcal{B}(\mathcal{K},(S\mu+3\kappa)  m |h|)}.
\end{aligned}
\end{equation}
Due to the similarity of estimates \cref{eq:uls} and \cref{eq:uls_n}, 
it is now possible to carry over the results of fixed-point iteration to Newton-Raphson iteration.
Here, the analogous estimates are given as
\begin{equation*}
\begin{aligned}
X^S
\leq & (e-1)(S\mu+\kappa)|h| \norm{\hat{\gRHS}-\gRHS{}}_{\mathcal{B}(\mathcal{K},(S\mu+3\kappa)  m |h|)},\\
\norm{\hat{\gRHS}-\gRHS{}}_{\mathcal{K}} \leq& \frac{\norm{(\Phi_{h,\hat{\gRHS}}^{L})^S-(\Phi_{h,\gRHS{}}^{L})^S}_{\mathcal{K}}}{S|h|} +\frac{(e-1)(\mu + \kappa/S) |h|\norm{\hat{\gRHS}-\gRHS{}}_{\mathcal{B}(\mathcal{K},(S\mu + 3\kappa)h_1 m)}}{h_1-|h|}.
\end{aligned}
\end{equation*}
The proof is completed.
\end{proof}

\subsection{Proof of Lemma \ref{lem:loss explain} (\nameref{lem:loss explain})}\label{app:lem:loss explain}
In the following, we seek to prove
a double inequality of the form $c_1 A \leq B \leq c_2 A$,
and, broadly speaking,
do this by showing (1) that $B \leq c_2 A$, and
(2) that $A \leq c_3 B$ with $c_1= 1/c_3$.
\begin{proof}
Denote by $C_L$ the Lipschitz constant of
$\brac{\Phi_{h,\fRHS{}_{\theta}}^{L}}^s$,
we have
\begin{equation*}
\begin{aligned}
&\sum_{x\in\mathcal{T}}\norm{\brac{\Phi_{h,\fRHS{}_{\theta}}^{L}}^s(\vec{x}) -\phi_{sh,\fRHS{}}(\vec{x})}_2^2 = \sum_{n=1}^N\sum_{m=1}^{M}\norm{\brac{\Phi_{h,\fRHS{}_{\theta}}^{L}}^s\circ\phi_{(m-1)\Delta t,\fRHS{}}(\vec{x}_n)-\phi_{m\Delta t, \fRHS{}}(\vec{x}_n)}_2^2\\
\leq & \sum_{n=1}^N\sum_{m=1}^{M}2 \norm{\brac{\Phi_{h,\fRHS{}_{\theta}}^{L}}^s\circ\phi_{(m-1)\Delta t, \fRHS{}}(\vec{x}_n)-\brac{\Phi_{h,\fRHS{}_{\theta}}^{L}}^{ms}}_2^2 + 2 \norm{\brac{\Phi_{h,\fRHS{}_{\theta}}^{L}}^{ms}(\vec{x}_n)-\phi_{m\Delta t, \fRHS{}}(\vec{x}_n)}_2^2\\
\leq &2(C_L^2+1)\cdot M^2\cdot \sum_{n=1}^N\sum_{m=1}^{M}\norm{\brac{\Phi_{h,\fRHS{}_{\theta}}^{L}}^{ms}(\vec{x}_n)-\phi_{m\Delta t, \fRHS{}}(\vec{x}_n)}_2^2/m^2.
\end{aligned}
\end{equation*}
In addition,
\begin{equation*}
\begin{aligned}
&\norm{\brac{\Phi_{h,\fRHS{}_{\theta}}^{L}}^{ms}(\vec{x}_n)-\phi_{m\Delta t, \fRHS{}}(\vec{x}_n)}_2^2/m^2\\
\leq& \sum_{i=0}^{m-1}\norm{ \brac{\Phi_{h,\fRHS{}_{\theta}}^{L}}^{(m-i)s}\circ\phi_{i\Delta t, \fRHS{}}(\vec{x}_n)-\brac{\Phi_{h,\fRHS{}_{\theta}}^{L}}^{(m-i-1)s}\circ \phi_{(i+1)\Delta t, \fRHS{}}(\vec{x}_n)}_2^2\\
\leq & \sum_{i=0}^{m-1} C_L^{2(m-i-1)}\norm{ \brac{\Phi_{h,\fRHS{}_{\theta}}^{L}}^{s}\circ\phi_{i\Delta t, \fRHS{}}(\vec{x}_n)-\phi_{(i+1)\Delta t, \fRHS{}}(\vec{x}_n)}_2^2\\
\leq & \sum_{i=1}^{M} C_L^{2(M-1)}\norm{ \brac{\Phi_{h,\fRHS{}_{\theta}}^{L}}^{s}\circ \phi_{(i-1)\Delta t, \fRHS{}}(\vec{x}_n)-\phi_{i\Delta t, \fRHS{}}(\vec{x}_n)}_2^2.
\end{aligned}
\end{equation*}
Therefore, we conclude that
\begin{equation}
\label{eq:ShootingVsTeacherProof}
\begin{aligned}
&\sum_{n=1}^N\sum_{m=1}^{M}\norm{\brac{\Phi_{h,\fRHS{}_{\theta}}^{L}}^{ms}(\vec{x}_n)-\phi_{m\Delta t, \fRHS{}}(\vec{x}_n)}_2^2/m^2 \\
\leq& \sum_{n=1}^N\sum_{m=1}^{M}\sum_{i=1}^{M} C_L^{2(M-1)}\norm{ \brac{\Phi_{h,\fRHS{}_{\theta}}^{L}}^{s}\circ \phi_{(i-1)\Delta t, \fRHS{}}(\vec{x}_n)-\phi_{i\Delta t, \fRHS{}}(\vec{x}_n)}_2^2\\
\leq & C_L^{2(M-1)} \cdot M \cdot \sum_{x\in\mathcal{T}}\norm{\brac{\Phi_{h,\fRHS{}_{\theta}}^{L}}^s(\vec{x}) -\phi_{sh,\fRHS{}}(\vec{x})}_2^2.
\end{aligned}
\end{equation}
The proof is completed.
\end{proof}

\subsection{Proof of Theorem \ref{the:hat} (\nameref{the:hat})}\label{app:the:hat}

We first demonstrate the convergence of both fixed-point iteration \cref{eq:irk} and Newton-Raphson iteration \cref{eq:irk_nr} for multiple compositions, which will be also used for the proof of \cref{lem:fp}. 
\begin{namedLemma}[\textbf{Multiple compositions of fixed-point iteration converges.}]
\label{lem:fp0}
Consider a consistent implicit Runge-Kutta scheme $\Phi_h$ \cref{eq:rk} and its approximation using fixed-point iteration $\Phi_h^{L}$ \cref{eq:irk}. Denote
\begin{equation*}
\mu =  \sum_{i=1}^I|b_i| \quad \kappa = \max_{1\leq i\leq I}\sum_{j=1}^I|a_{ij}|.
\end{equation*}
Then, for any continuously differentiable $\gRHS{}$ and initial value $\vec{x}$, there exists remainder term $R = \mathcal{O}(h^{L+3})$ such that
\begin{equation*}
\norm{\big(\Phi_{h,\gRHS{}}^{L}\big)^S(\vec{x}) -\big(\Phi_{h,\gRHS{}}\big)^S(\vec{x})}_{\infty} \leq S \norm{\gRHS{}(\vec{x})} \norm{\gRHS{}'(\vec{x})}^{L+1}\mu \kappa^{L+1} h^{L+2} + R.
\end{equation*}
\end{namedLemma}
\begin{proof}
The solution of $\big(\Phi_{h,\gRHS{}}\big)^S(\vec{x})$ and $\big(\Phi_{h,\gRHS{}}^{L}\big)^S(\vec{x})$ with initial value $\vec{x}$ are respectively given by
\begin{equation*}
\left\{\begin{aligned}
&\vec{x}^0 = \vec{x},\\
&\vec{v}_{i}^{s} = \vec{x}^{s} + h\sum_{j=1}^I a_{ij} \gRHS{}(\vec{v}_{j}^{s}),\\
&\vec{x}^{s+1} = \vec{x}^{s} + h\sum_{i=1}^I b_{i} \gRHS{}(\vec{v}_i^{s}),\\
\end{aligned}\right.
\quad
\left\{\begin{aligned}
&\vec{x}^{L,0} = \vec{x},\quad \vec{v}_i^{0,s} = \vec{x}^{L,s},\\
&\vec{v}_i^{l,s} = \vec{x}^{L,s} + h\sum_{j=1}^I a_{ij}  \gRHS{}(\vec{v}_{j}^{l-1,s}),\\
&\vec{x}^{L,s+1} = \vec{x}^{L,s} + h\sum_{i=1}^I b_{i} \gRHS{}(\vec{v}_i^{L,s}),\\
\end{aligned}\right.
\end{equation*}
where $s = 0,\cdots, S-1$, $l =1, \cdots, L$, $i=1,\cdots, I$ and $\big(\Phi_{h,\gRHS{}}\big)^S(\vec{x}) = \vec{x}^S$, $\big(\Phi_{h,\gRHS{}}^{L}\big)^S(\vec{x}) = \vec{x}^{L,S}$. Denote $V^{l,s} = \max_{i}\norm{\vec{v}_i^{l.s} - \vec{v}_i^s}$, we have
\begin{equation*}
\begin{aligned}
\norm{\vec{v}_i^{l.s} - \vec{v}_i^s} \leq &   m_1 \kappa h \cdot V^{l-1,s}+ \norm{\vec{x}^{L,s} - \vec{x}^s},
\end{aligned}
\end{equation*}
where $m_1 = \max_{s,l,i} \norm{\gRHS{}(\vec{v}_i^s) - \gRHS{}(\vec{v}_i^{l,s})}/\norm{\vec{v}_i^s - \vec{v}_i^{l,s}}$. As a result,
\begin{equation*}
\begin{aligned}
V^{L,s} \leq& (m_1 \kappa h)^L \cdot V^{0,s} + \frac{1-(m_1 \kappa h)^L}{1-m_1 \kappa h} \norm{\vec{x}^{L,s} - \vec{x}^s} \\
\leq& (m_1 \kappa h)^{L} \cdot m_0 \kappa h  + \frac{1}{1-m_1 \kappa h} \norm{\vec{x}^{L,s} - \vec{x}^s},
\end{aligned}
\end{equation*}
where $m_0 = \max_{s,i}{|\gRHS{}(\vec{v}_i^s)|}$. In addition,  we deduce that
\begin{equation*}
\begin{aligned}
\norm{\vec{x}^{L,s+1} - \vec{x}^{s+1}} \leq& \norm{\vec{x}^{L,s} - \vec{x}^s} + m_1\mu h \cdot V^{L,s}\\
\leq& (1+ \frac{m_1\mu h}{1-m_1 \kappa h}) \norm{\vec{x}^{L,s} - \vec{x}^s}  +  m_0m_1^{L+1} \mu \kappa^{L+1} h^{L+2}.
\end{aligned}
\end{equation*}
Finally, we obtain that
\begin{equation*}
\begin{aligned}
\norm{\vec{x}^{L,S} - \vec{x}^{S}} \leq & (1+ \frac{m_1\mu h}{1-m_1 \kappa h})^S \norm{\vec{x}^{L,0} - \vec{x}^{0}} + \frac{(1+ \frac{m_1\mu h}{1-m_1 \kappa h})^S-1}{\frac{m_1 \mu h}{1-m_1 \kappa h}}m_1^{L+1} m_0\mu \kappa^{L+1} h^{L+2}\\
\leq & \frac{(1+ \frac{m_1 \mu h}{1-m_1 \kappa h})^S-1}{\frac{m_1 \mu h}{1-m_1 \kappa h}}\cdot m_0m_1^{L+1}\mu \kappa^{L+1} h^{L+2}\\
=& S \norm{\gRHS{}(\vec{x})} \norm{\gRHS{}'(\vec{x})}^{L+1}\mu \kappa^{L+1} h^{L+2} + \mathcal{O}(h^{L+3}).
\end{aligned}
\end{equation*}
The proof is complete.
\end{proof}

\begin{namedLemma}[\textbf{Multiple compositions of Newton-Raphson iteration converges.}]
\label{lem:nr0}
Consider a consistent implicit Runge-Kutta scheme $\Phi_h$ \cref{eq:rk} and its approximation using Newton-Raphson iteration $\Phi_h^{L}$ \cref{eq:irk_nr}. Then, for any twice continuously differentiable $\gRHS{}$ and initial value $\vec{x}$, there exist remainder term $R = \mathcal{O}(h^{2^{L+1}+1})$ such that
\begin{equation*}
\norm{\big(\Phi_{h,\gRHS{}}^{L}\big)^S(\vec{x}) -\big(\Phi_{h,\gRHS{}}\big)^S(\vec{x})}_{\infty} \leq S \norm{\gRHS{}'(\vec{x})}(\frac{\norm{\gRHS{}''(\vec{x})}}{1-\norm{\gRHS{}'(\vec{x})}\kappa h})^{2^{L}-1} \norm{\gRHS{}(\vec{x})}^{2^{L}} \mu\kappa ^{2^{L+1}-1} h^{2^{L+1}} + R,
\end{equation*}
where $\mu$ and $\kappa$ are constants defined in \cref{lem:fp0}.
\end{namedLemma}
\begin{proof}
The solution of $\big(\Phi_{h,\gRHS{}}\big)^S(\vec{x})$ and $\big(\Phi_{h,\gRHS{}}^{L}\big)^S(\vec{x})$ with initial value $\vec{x}$ are respectively given by
\begin{equation*}
\left\{\begin{aligned}
&\vec{x}^0 = \vec{x},\\
&\vec{v}_{i}^{s} = \vec{x}^{s} + h\sum_{j=1}^I a_{ij} \gRHS{}(\vec{v}_{j}^{s}),\\
&\vec{x}^{s+1} = \vec{x}^{s} + h\sum_{i=1}^I b_{i} \gRHS{}(\vec{v}_i^{s}),\\
\end{aligned}\right.
\quad
\left\{\begin{aligned}
&\vec{x}^{L,0} = \vec{x},\quad \vec{v}_i^{0,s} = \vec{x}^{L,s},\\
&\vec{v}_i^{l,s} = \vec{x}^{L,s} + h\sum_{j=1}^I a_{ij} \big(\gRHS{}(\vec{v}_{j}^{l-1,s})+ g'(\vec{v}_{j}^{l-1,s})(\vec{v}_{j}^{l,s} - \vec{v}_{j}^{l-1,s})\big),\\
&\vec{x}^{L,s+1} = \vec{x}^{L,s} + h\sum_{i=1}^I b_{i} \gRHS{}(\vec{v}_i^{L,s}),\\
\end{aligned}\right.
\end{equation*}
where $s = 0,\cdots, S-1$, $l =1, \cdots, L$, $i=1,\cdots, I$ and $\big(\Phi_{h,\gRHS{}}\big)^S(\vec{x}) = \vec{x}^S$, $\big(\Phi_{h,\gRHS{}}^{L}\big)^S(\vec{x}) = \vec{x}^{L,S}$. Let
\begin{equation*}
\vec{\hat{v}}_i^s = \vec{x}^{L,s} + h\sum_{j=1}^I a_{ij} \gRHS{}(\vec{\hat{v}}_{j}^{s}), \ \text{for}\ i= 1,\cdots, I,\ s=0,\cdots, S-1,
\end{equation*}
and $m_0 = \max_{s,i}{|\gRHS{}(\vec{\hat{v}}_i^s)|}$, $m_1 = \max_{s,l,i} \norm{\gRHS{}'(\vec{v}_{j}^{l-1,s})}$, $m_2 = \max_{s,l,i} \sup_{\theta\in [0,1]}\norm{\gRHS{}''(\theta \vec{v}_{j}^{l-1,s} + (1-\theta)\vec{\hat{v}}_j^s)}/2$, $\hat{V}^{l,s} = \max_{i}\norm{\vec{v}_i^{l.s} - \vec{\hat{v}}_i^s}$, we have that
\begin{equation}\label{eq:newton 1}
\begin{aligned}
\norm{\vec{v}_i^{l.s} - \vec{\hat{v}}_i^s} \leq &  h\sum_{j=1}^I |a_{ij}| \norm{\gRHS{}(\vec{v}_{j}^{l-1,s})+ g'(\vec{v}_{j}^{l-1,s})(\vec{\hat{v}}_j^s - \vec{v}_{j}^{l-1,s})-\gRHS{}(\vec{\hat{v}}_j^s) + g'(\vec{v}_{j}^{l-1,s})(\vec{v}_{j}^{l,s} - \vec{\hat{v}}_j^s)}\\
\leq & m_2\kappa h (\hat{V}^{l-1,s})^2 + m_1\kappa h \hat{V}^{l,s},
\end{aligned}
\end{equation}
which implies that $\hat{V}^{l,s} \leq (\frac{m_2\kappa h}{1-m_1\kappa h})^{2^l-1} (m_0 \kappa h)^{2^l}$.
Let $\tilde{m}_1 = \max_{s,i} \norm{\gRHS{}(\vec{\hat{v}}_i^s) - \gRHS{}(\vec{v}_i^s)}/\norm{\vec{\hat{v}}_i^s - \vec{v}_i^s}$ and $\tilde{V}^{s} = \max_{i}\norm{\vec{\hat{v}}_i^s-\vec{v}_i^s}$, we have that
\begin{equation*}
\norm{\vec{\hat{v}}_i^s - \vec{v}_i^s} \leq \norm{\vec{x}^{L,s}- \vec{x}^s} + \tilde{m}_1\kappa h \tilde{V}^s,
\end{equation*}
which implies that $\tilde{V}^{s} \leq  \norm{\vec{x}^{L,s}- \vec{x}^s} /(1-\tilde{m}_1\kappa h)$. Therefore, we conclude that
\begin{equation*}
V^{l,s} \leq \hat{V}^{l,s} + \tilde{V}^{s} \leq (\frac{m_2\kappa h}{1-m_1\kappa h})^{2^l-1} (m_0 \kappa h)^{2^l} + \frac{1}{1- \tilde{m}_1\kappa h}\norm{\vec{x}^{L,s}- \vec{x}^s}.
\end{equation*}

In addition, similarly to \cref{lem:fp0}, we have that
\begin{equation*}
\begin{aligned}
\norm{\vec{x}^{L,s+1} - \vec{x}^{s+1}} \leq & \norm{\vec{x}^{L,s} - \vec{x}^s} + m_1\mu h \cdot V^{L,s}\\
\leq & (1+ \frac{m_1\mu h }{1-\tilde{m}_1 \kappa h}) \norm{\vec{x}^{L,s} - \vec{x}^s}  +  m_1(\frac{m_2}{1-m_1\kappa h})^{2^{L}-1} m_0^{2^{L}}\mu \kappa^{2^{L+1}-1} h^{2^{L+1}}.
\end{aligned}
\end{equation*}
and thus
\begin{equation*}
\begin{aligned}
&\norm{\vec{x}^{L,S} - \vec{x}^{S}}\\
\leq & (1+ \frac{m_1\mu h}{1-\tilde{m}_1  \kappa h})^S\norm{\vec{x}^{L,0} - \vec{x}^{0}} +
\frac{(1+ \frac{m_1\mu h}{1-\tilde{m}_1  \kappa h})^S-1}{\frac{m_1 \mu h}{1-\tilde{m}_1  \kappa h}} m_1(\frac{m_2}{1-m_1\kappa h})^{2^{L}-1} m_0^{2^{L}}\mu \kappa^{2^{L+1}-1} h^{2^{L+1}}\\
\leq & S \norm{\gRHS{}'(\vec{x})}(\frac{\norm{\gRHS{}''(\vec{x})}}{1-\norm{\gRHS{}'(\vec{x})}\kappa h})^{2^{L}-1} \norm{\gRHS{}(\vec{x})}^{2^{L}} \mu\kappa ^{2^{L+1}-1} h^{2^{L+1}} + \mathcal{O}(h^{2^{L+1}+1}),
\end{aligned}
\end{equation*}
which completes the proof.
\end{proof}

We next present the proof of \cref{the:hat}.
\begin{proof}[Proof of \cref{the:hat}]
We first prove that the statement holds for fixed-point iteration by induction. First, the case when $k=0$ is obvious since $\fRHS{} = \hat{\fRHS}_0=\fRHS{}_{0}$. Suppose now that $\hat{\fRHS}_k=\fRHS{}_{k}$ for $0\leq k \leq K \leq L-1$, then
\begin{equation*}
\hat{\fRHS}_h^{K} =\sum_{k=0}^K h^k \hat{\fRHS}_k = \sum_{k=0}^K h^k \fRHS{}_{k}= \fRHS{}_{h}^{K}
\end{equation*}
By \cref{lem:fp0}, we have
\begin{equation}\label{eq:diff}
\Phi_{h,\hat{\fRHS}_h^{K}}- \Phi^L_{h,\hat{\fRHS}_h^{K}} = \mathcal{O}(h^{L+2}).
\end{equation}
We rewrite the calculation procedure of IMDE as
\begin{equation*}
\begin{aligned}
\phi_{h,\fRHS{}} - \Phi_{h,\hat{\fRHS}_h^{K}} = & h^{K+2}\hat{\fRHS}_{K+1} +\mathcal{O}(h^{K+3}),\\
\phi_{h,\fRHS{}} - \Phi^L_{h,\fRHS{}_{h}^{K}} = & h^{K+2}\fRHS{}_{K+1} +\mathcal{O}(h^{K+3}).\\
\end{aligned}
\end{equation*}
Subtracting above two equations and substituting \cref{eq:diff}, we conclude that $\hat{\fRHS}_{K+1}=\fRHS{}_{K+1}$, which completes the induction.

In addition, for Newton-Raphson iteration, by \cref{lem:nr0}, repeating the above induction implies $\hat{\fRHS}_k=\fRHS{}_{k}$ for $0\leq k \leq 2^{L+1}-2$. The proof is completed.
\end{proof}

\subsection{Proof of Thereom \ref{the:error} (\nameref{the:error})}\label{app:the:error}

\begin{proof}
The proof is a direct consequence of \cref{the:implicit imde}, \cref{the:hat} and the following Lemma.
\end{proof}

\begin{namedLemma}[\textbf{IMDE power series for a $p^{\text{th}}$ order integrator has first error term of order $h^p$.}]
\label{lem:order}
Suppose that the integrator $\Phi_{h}(\vec{x})$ with discrete step $h$ is of order $p\geq 1$, then, the IMDE obeys
\begin{equation*}
\frac{d}{dt}\vec{\tilde{y}}=\fRHS{}_h(\vec{\tilde{y}})=\fRHS{}(\vec{\tilde{y}})+h^p\fRHS{}_{p}(\vec{\tilde{y}})+\cdots.
\end{equation*}
\end{namedLemma}
\begin{proof}
The proof can be found in~\cite{zhu2022on}.
\end{proof}

\subsection{Proof of Lemma \ref{lem:fp} (\nameref{lem:fp})}\label{app:lem:fp}

\begin{proof}
By Minkowski's inequality, we obtain that for neural network $\fRHS{}_{\theta}$,
\begin{equation*}
\begin{aligned}
\mathcal{L}_{exact}^{\frac{1}{2}}\leq & \mathcal{L}_{unrolled}^{\frac{1}{2}}+\mathcal{R}_{L}, \\
\mathcal{L}_{exact}^{\frac{1}{2}}\leq & \mathcal{L}_{unrolled}^{\frac{1}{2}}+\left(\sum_{n=1}^N\sum_{m=1}^{M} \norm{\brac{\Phi_{h,\fRHS{}_{\theta}}^{L}(\vec{x}_n)}^{ms}- \brac{\Phi_{h,\fRHS{}_{\theta}}^{L+1}(\vec{x}_n)}^{ms}}_2^2/(m\Delta t)^2\right)^{\frac{1}{2}} + \mathcal{R}_{L+1},
\end{aligned}
\end{equation*}
where 
\begin{equation*}
\mathcal{R}_{L} = \left(\sum_{n=1}^N\sum_{m=1}^{M} \norm{\brac{\Phi_{h,\fRHS{}_{\theta}}^L(\vec{x}_n)}^{ms}- \brac{\Phi_{h,\fRHS{}_{\theta}}(\vec{x}_n)}^{ms}}_2^2/(m\Delta t)^2\right)^{\frac{1}{2}}.    
\end{equation*}
According to \cref{lem:fp0} and \cref{lem:nr0},  we have $\mathcal{R}_{L} =\mathcal{O}(h^{L^*+1})$ where $L^{*}=L$ for the unrolled approximation using fixed-point iteration \cref{eq:irk} and $L^{*}=2^{L+1}-2$ for the unrolled approximation using Newton-Raphson iteration \cref{eq:irk_nr} and thus complete the proof.
\end{proof}


\bibliographystyle{siamplain}
\bibliography{references}

\begin{thebibliography}{10}

\bibitem{Almeida1987}
{\sc L.~B. Almeida}, {\em A learning rule for asynchronous perceptrons with
  feedback in a combinatorial environment}, in IEEE First International
  Conference on Neural Networks, IEEE, 1987, pp.~608--618.

\bibitem{anderson1996comparison}
{\sc J.~Anderson, I.~Kevrekidis, and R.~Rico-Martinez}, {\em A comparison of
  recurrent training algorithms for time series analysis and system
  identification}, Computers \& chemical engineering, 20 (1996),
  pp.~S751--S756.

\bibitem{bai2019deep}
{\sc S.~Bai, J.~Z. Kolter, and V.~Koltun}, {\em Deep equilibrium models}, in
  33rd Conference on Neural Information Processing Systems (NeurIPS 2019),
  Vancouver, BC, Canada, 2019, pp.~688--699.

\bibitem{bai2022neural}
{\sc S.~Bai, V.~Koltun, and J.~Z. Kolter}, {\em Neural deep equilibrium
  solvers}, in International Conference on Learning Representations, 2022.

\bibitem{behrmann2019invertible}
{\sc J.~Behrmann, W.~Grathwohl, R.~T.~Q. Chen, D.~Duvenaud, and J.~Jacobsen},
  {\em Invertible residual networks}, in Proceedings of the 36th International
  Conference on Machine Learning, {ICML} 2019, Long Beach, California, {USA},
  vol.~97, {PMLR}, 2019, pp.~573--582.

\bibitem{bertalan2019learning}
{\sc T.~Bertalan, F.~Dietrich, I.~Mezi{\'c}, and I.~G. Kevrekidis}, {\em On
  learning hamiltonian systems from data}, Chaos: An Interdisciplinary Journal
  of Nonlinear Science, 29 (2019), p.~121107.

\bibitem{botev2021priors}
{\sc A.~Botev, A.~Jaegle, P.~Wirnsberger, D.~Hennes, and I.~Higgins}, {\em
  Which priors matter? benchmarking models for learning latent dynamics}, in
  35th Conference on Neural Information Processing Systems (NeurIPS 2021) Track
  on Datasets and Benchmarks, 2021.

\bibitem{brunton2016discovering}
{\sc S.~L. Brunton, J.~L. Proctor, and J.~N. Kutz}, {\em Discovering governing
  equations from data by sparse identification of nonlinear dynamical systems},
  Proceedings of the national academy of sciences, 113 (2016), pp.~3932--3937.

\bibitem{chartier2010algebraic}
{\sc P.~Chartier, E.~Hairer, and G.~Vilmart}, {\em Algebraic structures of
  b-series}, Foundations of Computational Mathematics, 10 (2010), pp.~407--427.

\bibitem{chen2021data}
{\sc R.~Chen and M.~Tao}, {\em Data-driven prediction of general hamiltonian
  dynamics via learning exactly-symplectic maps}, in Proceedings of the 38th
  International Conference on Machine Learning ({ICML} 2021), vol.~139, {PMLR},
  2021, pp.~1717--1727.

\bibitem{chen2018neural}
{\sc T.~Chen, Y.~Rubanova, J.~Bettencourt, and D.~Duvenaud}, {\em Neural
  ordinary differential equations}, in 32nd Conference on Neural Information
  Processing Systems (NeurIPS 2018), 2018, pp.~6572--6583.

\bibitem{chen2020symplectic}
{\sc Z.~Chen, J.~Zhang, M.~Arjovsky, and L.~Bottou}, {\em Symplectic recurrent
  neural networks}, in 8th International Conference on Learning
  Representations, {ICLR} 2020, Addis Ababa, Ethiopia, 2020.

\bibitem{daniels2014efficient}
{\sc B.~C. Daniels and I.~Nemenman}, {\em Efficient inference of parsimonious
  phenomenological models of cellular dynamics using s-systems and alternating
  regression}, Plos One, 10 (2014).

\bibitem{doncevic2022recursively}
{\sc D.~T. Doncevic, A.~Mitsos, Y.~Guo, Q.~Li, F.~Dietrich, M.~Dahmen, and
  I.~G. Kevrekidis}, {\em A recursively recurrent neural network ({{R2N2}})
  architecture for learning iterative algorithms}, 2022,
  \url{https://arxiv.org/abs/2211.12386}.

\bibitem{du2022discovery}
{\sc Q.~Du, Y.~Gu, H.~Yang, and C.~Zhou}, {\em The discovery of dynamics via
  linear multistep methods and deep learning: error estimation}, SIAM Journal
  on Numerical Analysis, 60 (2022), pp.~2014--2045.

\bibitem{eirola1993aspects}
{\sc T.~Eirola}, {\em Aspects of backward error analysis of numerical
  {{ODE}}s}, Journal of Computational and Applied Mathematics, 45 (1993),
  pp.~65--73.

\bibitem{laurent2021implicit}
{\sc L.~El~Ghaoui, F.~Gu, B.~Travacca, A.~Askari, and A.~Tsai}, {\em Implicit
  deep learning}, SIAM Journal on Mathematics of Data Science, 3 (2021),
  pp.~930--958.

\bibitem{feng1991formal}
{\sc K.~Feng}, {\em Formal power series and numerical algorithms for dynamical
  systems}, in Proceedings of international conference on scientific
  computation, Hangzhou, China, Series on Appl. Math. Singapore: World
  Scientific, vol.~1, 1991, pp.~28--35.

\bibitem{feng1993formal}
{\sc K.~Feng}, {\em Formal dynamical systems and numerical algorithms}, SERIES
  ON APPLIED MATHEMATICS, 4 (1993), pp.~1--10.

\bibitem{geng2021on}
{\sc Z.~Geng, X.-Y. Zhang, S.~Bai, Y.~Wang, and Z.~Lin}, {\em On training
  implicit models}, in Advances in Neural Information Processing Systems, 2021.

\bibitem{gonzalez1998identification}
{\sc R.~Gonz{\'a}lez-Garc{\'\i}a, R.~Rico-Mart{\`\i}nez, and I.~G. Kevrekidis},
  {\em Identification of distributed parameter systems: A neural net based
  approach}, Computers \& chemical engineering, 22 (1998), pp.~S965--S968.

\bibitem{greydanus2019hamiltonian}
{\sc S.~Greydanus, M.~Dzamba, and J.~Yosinski}, {\em {Hamiltonian} neural
  networks}, in Advances in Neural Information Processing Systems 32, 2019,
  pp.~15353--15363.

\bibitem{Guo2021PersonalizedAG}
{\sc Y.~Guo, F.~Dietrich, T.~S. Bertalan, D.~T. Doncevic, M.~Dahmen, I.~G.
  Kevrekidis, and Q.~Li}, {\em Personalized algorithm generation: A case study
  in learning {{ODE}} integrators}, SIAM J. Sci. Comput., 44 (2021),
  pp.~1911--.

\bibitem{hairer1997life}
{\sc E.~Hairer and C.~Lubich}, {\em The life-span of backward error analysis
  for numerical integrators}, Numerische Mathematik, 76 (1997), pp.~441--462.

\bibitem{hairer2006geometric}
{\sc E.~Hairer, C.~Lubich, and G.~Wanner}, {\em Geometric numerical
  integration: structure-preserving algorithms for ordinary differential
  equations}, vol.~31, Springer Science \& Business Media, 2006.

\bibitem{hairer1996solving}
{\sc E.~Hairer and G.~Wanner}, {\em Solving ordinary differential equations
  II}, vol.~375, Springer Berlin Heidelberg, 1996.

\bibitem{hu2022revealing}
{\sc P.~Hu, W.~Yang, Y.~Zhu, and L.~Hong}, {\em Revealing hidden dynamics from
  time-series data by {{ODEN}}et}, Journal of Computational Physics, 461
  (2022), p.~111203.

\bibitem{huang2021implicit}
{\sc Z.~Huang, S.~Bai, and J.~Z. Kolter}, {\em {$($Implicit$)^{2}$}: Implicit
  layers for implicit representations}, 35th Conference on Neural Information
  Processing Systems (NeurIPS 2021), 34 (2021).

\bibitem{huh2020time}
{\sc I.~Huh, E.~Yang, S.~J. Hwang, and J.~Shin}, {\em Time-reversal symmetric
  {ODE} network}, in 34th Conference on Neural Information Processing Systems
  (NeurIPS 2020), 2020.

\bibitem{jin2020sympnets}
{\sc P.~Jin, Z.~Zhang, A.~Zhu, Y.~Tang, and G.~E. Karniadakis}, {\em Sympnets:
  Intrinsic structure-preserving symplectic networks for identifying
  hamiltonian systems}, Neural Networks, 132 (2020), pp.~166--179.

\bibitem{keller2021discovery}
{\sc R.~T. Keller and Q.~Du}, {\em Discovery of dynamics using linear multistep
  methods}, SIAM Journal on Numerical Analysis, 59 (2021), pp.~429--455.

\bibitem{kingma2014adam}
{\sc D.~P. Kingma and J.~Ba}, {\em Adam: {A} method for stochastic
  optimization}, in 3rd International Conference on Learning Representations,
  2014.

\bibitem{lovelett2020graybox}
{\sc R.~J. Lovelett, J.~L. Avalos, and I.~G. Kevrekidis}, {\em Partial
  observations and conservation laws: Gray-box modeling in biotechnology and
  optogenetics}, Industrial \& Engineering Chemistry Research, 59 (2020),
  pp.~2611--2620, \url{https://doi.org/10.1021/acs.iecr.9b04507},
  \url{https://doi.org/10.1021/acs.iecr.9b04507},
  \url{https://arxiv.org/abs/https://doi.org/10.1021/acs.iecr.9b04507}.

\bibitem{lu2019nonparametric}
{\sc F.~Lu, M.~Zhong, S.~Tang, and M.~Maggioni}, {\em Nonparametric inference
  of interaction laws in systems of agents from trajectory data}, Proceedings
  of the National Academy of Sciences, 116 (2019), pp.~14424--14433.

\bibitem{pal2021opening}
{\sc A.~Pal, Y.~Ma, V.~B. Shah, and C.~V. Rackauckas}, {\em Opening the
  blackbox: Accelerating neural differential equations by regularizing internal
  solver heuristics}, in Proceedings of the 38th International Conference on
  Machine Learning, {ICML} 2021, vol.~139, {PMLR}, 2021, pp.~8325--8335.

\bibitem{pineda87_gener_back_propag_to_recur_neural_networ}
{\sc F.~J. Pineda}, {\em Generalization of back-propagation to recurrent neural
  networks}, Physical Review Letters, 59 (1987), pp.~2229--2232,
  \url{https://doi.org/10.1103/physrevlett.59.2229},
  \url{https://doi.org/10.1103/physrevlett.59.2229}.

\bibitem{poli2020hypersolver}
{\sc M.~Poli, S.~Massaroli, A.~Yamashita, H.~Asama, and J.~Park}, {\em
  Hypersolvers: Toward fast continuous-depth models}, in 34th Conference on
  Neural Information Processing Systems (NeurIPS 2020), Vancouver, Canada.,
  2020.

\bibitem{chapter_bulsari}
{\sc I.~G.~K. R.~Rico-Mart\'{i}nez and K.~Krischer}, {\em Nonlinear system
  identification using neural networks: dynamics and instabilities}, Elsevier
  Science, 1995, ch.~16.

\bibitem{raissi2018hidden}
{\sc M.~Raissi and G.~E. Karniadakis}, {\em Hidden physics models: Machine
  learning of nonlinear partial differential equations}, Journal of
  Computational Physics, 357 (2018), pp.~125--141.

\bibitem{raissi2018multistep}
{\sc M.~Raissi, P.~Perdikaris, and G.~E. Karniadakis}, {\em Multistep neural
  networks for data-driven discovery of nonlinear dynamical systems}, arXiv
  preprint arXiv:1801.01236,  (2018).

\bibitem{reich1999backward}
{\sc S.~Reich}, {\em Backward error analysis for numerical integrators}, SIAM
  Journal on Numerical Analysis, 36 (1999), pp.~1549--1570.

\bibitem{rico1994continuous}
{\sc R.~Rico-Martinez, J.~Anderson, and I.~Kevrekidis}, {\em Continuous-time
  nonlinear signal processing: a neural network based approach for gray box
  identification}, in Proceedings of IEEE Workshop on Neural Networks for
  Signal Processing, IEEE, 1994, pp.~596--605.

\bibitem{rico1993continuous}
{\sc R.~Rico-Martinez and I.~G. Kevrekidis}, {\em Continuous time modeling of
  nonlinear systems: A neural network-based approach}, in IEEE International
  Conference on Neural Networks, IEEE, 1993, pp.~1522--1525.

\bibitem{rico-martinez92_discr_vs}
{\sc R.~Rico-Mart{\'i}nez, K.~Krischer, I.~Kevrekidis, M.~Kube, and J.~Hudson},
  {\em Discrete- vs. continuous-time nonlinear signal processing of cu
  electrodissolution data}, Chemical Engineering Communications, 118 (1992),
  pp.~25--48, \url{https://doi.org/10.1080/00986449208936084}.

\bibitem{saad1986gmres}
{\sc Y.~Saad and M.~H. Schultz.}, {\em {GMRES: A generalized minimal residual
  algorithm for solving nonsymmetric linear systems}}, SIAM Journal on
  scientific and statistical computing, 7 (1986), pp.~856--869.

\bibitem{sanz1992symplectic}
{\sc J.~M. Sanz-Serna}, {\em Symplectic integrators for hamiltonian problems:
  an overview}, Acta numerica, 1 (1992), pp.~243--286.

\bibitem{schmidt2009distilling}
{\sc M.~Schmidt and H.~Lipson}, {\em Distilling free-form natural laws from
  experimental data}, Science, 324 (2009), pp.~81--85.

\bibitem{toth2020hamiltonian}
{\sc P.~Toth, D.~J. Rezende, A.~Jaegle, S.~Racani{\`{e}}re, A.~Botev, and
  I.~Higgins}, {\em Hamiltonian generative networks}, in 8th International
  Conference on Learning Representations, {ICLR} 2020, Addis Ababa, Ethiopia,
  2020.

\bibitem{williams1989learning}
{\sc R.~J. Williams and D.~Zipser}, {\em A learning algorithm for continually
  running fully recurrent neural networks}, Neural computation, 1 (1989),
  pp.~270--280.

\bibitem{wu2020structure}
{\sc K.~Wu, T.~Qin, and D.~Xiu}, {\em Structure-preserving method for
  reconstructing unknown hamiltonian systems from trajectory data}, SIAM
  Journal on Scientific Computing, 42 (2020), pp.~A3704--A3729.

\bibitem{wu2019numerical}
{\sc K.~Wu and D.~Xiu}, {\em Numerical aspects for approximating governing
  equations using data}, Journal of Computational Physics, 384 (2019),
  pp.~200--221.

\bibitem{yoshida1993recent}
{\sc H.~Yoshida}, {\em Recent progress in the theory and application of
  symplectic integrators}, Qualitative and Quantitative Behaviour of Planetary
  Systems,  (1993), pp.~27--43.

\bibitem{yu2021onsagernet}
{\sc H.~Yu, X.~Tian, E.~Weinan, and Q.~Li}, {\em Onsagernet: Learning stable
  and interpretable dynamics using a generalized onsager principle}, Physical
  Review Fluids, 6 (2021), p.~114402.

\bibitem{zhang2021gfinns}
{\sc Z.~Zhang, Y.~Shin, and G.~E. Karniadakis}, {\em Gfinns: Generic formalism
  informed neural networks for deterministic and stochastic dynamical systems},
  arXiv preprint arXiv:2109.00092,  (2021).

\bibitem{zhu2022on}
{\sc A.~Zhu, P.~Jin, B.~Zhu, and Y.~Tang}, {\em On numerical integration in
  neural ordinary differential equations}, in Proceedings of the 39th
  International Conference on Machine Learning ({ICML} 2022), vol.~162, {PMLR},
  2022, pp.~27527--27547.

\end{thebibliography}

\end{document}